\documentclass{amsart}
\usepackage[T1]{fontenc}
\usepackage{lmodern}
\usepackage{amsthm,amsmath,amsfonts,amssymb,bbm,mathrsfs}
\usepackage{mathtools}
\usepackage{enumitem}
\usepackage{url}
\usepackage{dsfont} 
\usepackage{appendix}
\usepackage{amsthm}
\usepackage{color}
\usepackage[dvipsnames]{xcolor}
\usepackage{graphicx}
\usepackage{float}
\usepackage[normalem]{ulem}

\usepackage[colorlinks=true, linkcolor=blue, urlcolor=black, citecolor=blue,pdfstartview=FitH]{hyperref}
\usepackage[english]{babel}
\usepackage{caption,tikz,subfigure}
\usetikzlibrary{shapes}
\usetikzlibrary{patterns}
\usepackage[top=3cm, bottom=3cm, left=3cm, right=3cm]{geometry}

\usetikzlibrary{decorations.pathreplacing,calc}
\usetikzlibrary{calc}
\usetikzlibrary{decorations.pathreplacing}
\makeatletter

\@addtoreset{equation}{section}
\makeatother

\setlist[enumerate,1]{label=(\roman*), font = \normalfont} 
\let\originalleft\left
\let\originalright\right
\renewcommand{\left}{\mathopen{}\mathclose\bgroup\originalleft}
\renewcommand{\right}{\aftergroup\egroup\originalright}

\newcommand{\N}{\mathbb{N}}
\newcommand{\Z}{\mathbb{Z}}

\newcommand{\R}{\mathbb{R}}
\newcommand{\C}{\mathbb{C}}
\newcommand{\T}{\mathbb{T}}
\renewcommand{\P}{\mathbb{P}}

\newcommand{\E}[1]{\mathbb{E} \left[#1\right]}
\newcommand{\Ec}[2]{\mathbb{E} \left[#1\middle|#2\right]}

\newcommand{\e}{e}
\newcommand{\ep}{\varepsilon}

\newcommand{\de}{\coloneqq}

\newcommand{\dt}{\mathrm{d}t}

\newcommand{\du}{\mathrm{d}u}

\newcommand{\dx}{\mathrm{d}x}

\newcommand{\F}{\mathrm F}
\newcommand{\ga}{\gamma}

\newcommand{\G}[1]{\Gamma\left(#1\right)}
\newcommand{\ind}{\mathds{1}}

\newcommand{\nor}[1]{\left\| #1 \right\|_\gamma}
\newcommand{\norm}[1]{\left\| #1 \right\|}

\newcommand{\ii}{\mathrm i}

\newcommand{\dz}{\mathrm d z}
\newcommand{\M}{\mathrm M_{\ii\beta}}
\newcommand{\dth}{\mathrm{d}\theta}
\newcommand{\dthp}{\mathrm{d}\theta'}
\newcommand{\gap}{\mathrm{gap}}

\theoremstyle{plain}
\newtheorem{thm}{Theorem}[section]
\newtheorem{prop}[thm]{Proposition}
\newtheorem{lem}[thm]{Lemma}
\newtheorem{cor}[thm]{Corollary}
\newtheorem{conj}[thm]{Conjecture}
\theoremstyle{definition}
\newtheorem{defi}[thm]{Definition}
\theoremstyle{remark}
\newtheorem{rem}[thm]{Remark}

\usepackage{anyfontsize}
\makeatletter

\def\l@section{\@tocline{1}{0pt}{0pt}{0pt}{}}          
\def\l@subsection{\@tocline{2}{0pt}{0pt}{2.3em}{}}     
\def\l@subsubsection{\@tocline{3}{0pt}{0pt}{5.3em}{}}  
\makeatother

\title{Fourier dimension of imaginary Gaussian multiplicative chaos}

\author{Benjamin Bonnefont}
\author{Hermanni Rajamäki}
\author{Vincent Vargas}
\thanks{All authors: Université de Genève, Section de mathématiques et département de physique théorique, 1211 Geneva, Switzerland.\\
	\textit{E-mails}: \href{mailto:benjamin.bonnefont@unige.ch}{benjamin.bonnefont@unige.ch}, \href{mailto:hermanni.rajamaki@unige.ch}{hermanni.rajamaki@unige.ch}, \href{mailto:vincent.vargas@unige.ch}{vincent.vargas@unige.ch}.}



\begin{document}

\begin{abstract}
	We study the high-frequency Fourier asymptotics of imaginary Gaussian multiplicative chaos on the unit circle, a complex-valued random distribution formally given by $\M=\exp(\ii\beta X)$, where $X$ is a log-correlated Gaussian field. In the subcritical phase $\beta\in(0,1)$, we prove that its Fourier dimension, defined by the optimal polynomial decay exponent of $|\widehat\M(n)|^2$, is almost surely equal to $1-\beta^2$. This result holds for a broad class of log-correlated fields whose covariance differs from the exact logarithmic kernel by a sufficiently regular function.
	
	For the exactly log-correlated field on the circle, we obtain the following results. We prove that the chaos almost surely fails to belong to $H^{-\beta^2/2}(\mathbb T)$, the critical Sobolev space left open by previous regularity results. We further establish a central limit theorem: the rescaled coefficients $n^{(1-\beta^2)/2}\widehat \M(n)$ converge in law to an isotropic complex Gaussian random variable, and finitely many consecutive coefficients converge jointly to independent copies. The high-frequency content of $\M$ behaves as a white noise: $n^{(1-\beta^2)/2}e^{\ii n\theta}\M$ converges in $H^s(\mathbb T)$, $s<-1/2$, to a complex white noise with explicit intensity $\kappa(\beta)=\frac{1}{\pi}\Gamma(1-\beta^2)
	\sin\big(\frac{\pi\beta^2}{2}\big)$.
	
	The proof relies on moment identities obtained from Coulomb-gas integrals and Jack-polynomial expansions. Their asymptotic analysis is governed by partitions with large gaps, where the Pieri coefficients appearing in these expansions simplify, and the leading contribution becomes explicit.

\end{abstract}

\maketitle

\section{Introduction}

Gaussian multiplicative chaos (GMC) is a theory of random measures, and
more generally random distributions, obtained by exponentiating log-correlated Gaussian fields. It was introduced by Kahane \cite{kahane85}, following earlier ideas of Mandelbrot \cite{mandelbrot74} on turbulence and H{\o}egh-Krohn \cite{hoeghkrohn71} in quantum field theory. Over the last decades, GMC has become a central object in probability and mathematical physics, with connections to turbulence, finance, conformal field theory, random matrix theory and related areas. Since log-correlated fields are random distributions rather than functions, the exponential must be defined through a regularization and renormalization procedure. In the real setting, this construction yields a canonical multifractal measure. The complex-valued theory is still much less explored, especially in the purely imaginary regime considered here.

In this work, we focus on the one-dimensional setting of the unit circle $\mathbb T=\mathbb R/2\pi\mathbb Z$. The underlying field is the trace of the two-dimensional Gaussian free field on  $\mathbb T$, with zero average. Equivalently, it is the centered Gaussian field  $X$ with covariance
\[
\E{X_{\theta}X_{\theta'}}
=
\log \frac{1}{\left| e^{\ii\theta}-e^{\ii\theta'}\right|},
\]
understood as a random distribution on $\mathbb T$ in the sense of Schwartz. If  $X_\varepsilon$ is a smooth convolution approximation of  $X$ on $\mathbb T$ and $\gamma\in\C$, we consider the renormalized exponentials
\begin{equation}\label{gmc}
	\mathrm M_{\gamma}^{\varepsilon}(\dth)
	\de
	\exp\left(\gamma X_\varepsilon(\theta)-\frac{\gamma^2}{2}\E{X_\varepsilon(\theta)^2}\right)\,\dth.
\end{equation}
The theory of GMC is concerned with the limit of  $\mathrm M_\gamma^{\varepsilon}$ as $\varepsilon\to0$, with convergence understood in a topology depending on the parameter $\gamma$.

For real  $\gamma$ with  $|\gamma|<\sqrt2$, this limit exists in probability for the weak topology of measures and is independent of the choice of mollifier; see for instance \cite{berestycki17}, and \cite{SHAMOV16} for a general uniqueness and convergence theorem. The limit is a non-trivial random measure, denoted by  $\mathrm M_\gamma$. It is carried by a set of Hausdorff dimension  $1-\gamma^2/2$, while its finer local behavior is multifractal \cite{rhodesvargas14,bertacco23}. For  $|\gamma|\ge \sqrt2$, the renormalization in \eqref{gmc} degenerates and different procedures are needed, see for instance \cite{lacoin24} for the critical case and \cite{bertaccohairer25} for the supercritical case.

One may also consider complex parameters  $\gamma=\alpha+\ii\beta$. In that case, the limit is no longer a positive measure in general, but rather a complex-valued random distribution. A discrete analogue, given by multiplicative cascades with complex weights was first studied in the physics literature in \cite{des93} and on the mathematical side in \cite{barraljinmandelbrot10}. In the one-dimensional setting considered here, the standard renormalization gives a non-trivial limit in the subcritical domain
\[
\mathcal D \de
\left\{|\gamma|<1\right\}\;\cup\;\left\{|\alpha|\in(1,\sqrt2) \text{ and } |\alpha|+|\beta|\leq\sqrt2\right\},
\]
see for instance \cite{lacoin22subchaos}. The behavior in other regions of the parameter space is related to further phases of complex multiplicative chaos and has also been studied in \cite{lacoin22,lrv2015}.

The present paper is concerned with the purely imaginary regime inside $\mathcal D$, namely $\gamma=\ii\beta$, with  $|\beta|<1$. In this case, \eqref{gmc} becomes
\[
\M^{\varepsilon}(\dth)
=
\exp\left(\ii\beta X_\varepsilon(\theta)+\frac{\beta^2}{2}\E{X_\varepsilon(\theta)^2}\right)\,\dth .
\]
This sequence does not converge to a measure. Instead, viewed as a sequence of random distributions, it converges in probability to a limit $\M$ in negative Sobolev spaces  $H^s(\mathbb T) $, for $s<-1/2$, see \cite{JunnilaSaksmanWebb20}. We recall that
\[
H^s(\mathbb T)
\de
\left\{\varphi\in\mathcal D'(\mathbb T):\norm{\varphi}_{H^s(\mathbb T)}^2
=
\sum_{n\in\mathbb Z}(1+n^2)^s\,|\widehat\varphi(n)|^2<\infty\right\},
\]
where, with our convention,
\[
\widehat\varphi(n)
\de
\frac{1}{2\pi}\left\langle \varphi, e^{-\ii n\theta}\right\rangle
\]
denotes the  $n$-th Fourier coefficient.

The imaginary case differs markedly from the real one. $\M$ is no longer a positive measure, nor a complex measure: it is a proper random distribution. Its total mass has finite moments of every order, unlike in the real case, where positive moments exist only up to a critical threshold. Imaginary chaos also exhibits a monofractal behavior \cite{arubaverezjegojunnila25}. For further properties and motivations, we refer to the original account \cite{JunnilaSaksmanWebb20} and to \cite{Arujegojunnila22} for density results. 

Imaginary Gaussian exponentials are closely related to classical Coulomb-gas and sine-Gordon models; see for instance \cite{Frohlich76} for the historical setting and \cite{lrv23} for a recent probabilistic construction. On the circle, the corresponding Coulomb-gas partition functions have a natural connection with Jack polynomials, already explored in the physics literature \cite{fendleylesagesaleur95}. Imaginary GMC also appears naturally in recent probabilistic approaches to imaginary Liouville theory \cite{UsciatiGuillarmouRhodesSantachiara26}.

Since $\M$ is a genuine complex-valued distribution rather than a positive measure, the usual geometric notions of dimension for measures do not apply directly. Fourier dimension therefore provides a natural harmonic counterpart in this setting. In the real case, this direction was initiated in \cite{garban2024harmonicanalysisgaussianmultiplicative}, where it was proved that the Fourier coefficients of $\mathrm M_\gamma$ tend to $0$ almost surely and sharp decay exponents were conjectured. These conjectures were recently established for multiplicative cascade models in \cite{chen2025harmonicanalysismandelbrotcascades} and for one-dimensional GMC in \cite{lin2025harmonicanalysismultiplicativechaos}. These results have since been extended to higher dimensional and other models \cite{lin2025harmonicanalysismultiplicativechaosII,chen2025exactvaluesfourierdimensions}. The behavior of the Fourier coefficients of critical real GMC has been investigated in \cite{arguin2026fouriercoefficientscriticalgaussian}. These methods rely crucially on the positivity of $\mathrm M_\gamma$ and on multifractal mass estimates, neither of which is available in the imaginary regime. Let us also mention that central limit theorems for the rescaled Fourier coefficients, analogous to those we obtain below, have been established for the holomorphic multiplicative chaos, introduced in \cite{najnudelpaquettesimm23} as a scaling limit of characteristic polynomials of unitary random matrices and further developed in \cite{najnudel2025fouriercoefficientsholomorphicmultiplicative, atherfold2025fouriercoefficientscriticalholomorphic}.

The goal of the present paper is to continue this harmonic analysis in the purely imaginary setting. We set
\begin{equation*}
	c_n \de \widehat{\M}(n).
\end{equation*}

The Fourier dimension quantifies the decay of the Fourier coefficients of $\varphi\in\mathcal D'(\mathbb T)$. This is an extension of the usual definition for measures to distributions and it is given by
\begin{equation*}
	\dim_{\mathrm F}(\varphi)\de \sup\left\{\, s\in[0,1] : |\widehat{\varphi}(n)|^2 = O(|n|^{-s}) \text{ as } |n|\to\infty \right\}.
\end{equation*}
Our first result identifies the Fourier dimension of $\M$.
\begin{thm}[Fourier dimension]\label{thm:1}
	For $\beta\in(0,1)$, almost surely,
	\[
	\dim_{\mathrm F}(\M) = 1-\beta^2.
	\]
\end{thm}
An upper bound on the Fourier dimension follows from the regularity results obtained in \cite{JunnilaSaksmanWebb20}. The authors prove that almost surely $\M\in H^s(\mathbb T)$ for $s<-\beta^2/2$ and $\M\notin H^s(\mathbb T)$ for $s>-\beta^2/2$, which implies $\dim_{\mathrm F}(\M)\le 1-\beta^2$. Indeed, any pointwise decay exponent strictly larger than $1-\beta^2$ would imply membership in some $H^s(\mathbb T)$ with $s>-\beta^2/2$. The content of Theorem \ref{thm:1} is to establish the lower bound. The sharp moment control used in the proof of the theorem allows us to determine the critical Sobolev regularity of $\M$.
\begin{cor}[Critical Sobolev regularity]\label{cor:sobcritic}
	Almost surely, $\M\notin H^{-\beta^2/2}(\T)$.
\end{cor}
An analogous result has been obtained for the imaginary chaos associated to the two-dimensional Gaussian free field in \cite{arubaverezjegojunnila25}. The proof of Theorem \ref{thm:1} relies primarily on the integrable structure provided by the exact logarithmic kernel through the Jack-polynomial expansion of the moments, see Section \ref{sec:mom1}. The Fourier-dimension statement itself, however, is not an artifact of this exact structure. We extend it to a class of underlying fields $X^g$ on the circle, with covariance
\[
K(\theta,\theta') \de \log\frac{1}{|e^{\ii\theta}-e^{\ii\theta'}|} + g(\theta,\theta'),
\]
assumed to be positive semi-definite. The corresponding imaginary chaos $\M^g$ is defined as in \eqref{gmc} from a regularization of $X^g$. We treat two regularity regimes: a milder assumption $g \in H^{1+\delta}(\mathbb T)$ in the stationary case where $g(\theta,\theta')$ is a function of $\theta-\theta'$ alone, and the stronger requirement $g \in H^{3/2+\delta}(\mathbb T^2)$ in the general case.

\begin{thm}[Robustness under smooth perturbations]\label{thm:robustness}
	If $g \in H^{1+\delta}(\mathbb T)$ for some $\delta>0$ in the stationary case, or $g\in H^{3/2+\delta}(\mathbb T^2)$ for some $\delta>0$ in the general case, then, almost surely
	\[
	\mathrm{dim}_{\mathrm F}(\M^g) = 1-\beta^2.
	\]	
\end{thm}

We now turn to finer asymptotics in the exactly log-correlated case, where the explicit moment identities yield a central limit theorem for the rescaled Fourier coefficients. We denote by $\mathcal N_\C(0,\sigma^2)$ the complex Gaussian random variable such that 
\[
\mathbb E\, \mathcal N_\C(0,\sigma^2)^2=0 \quad
\text{ and }
\quad
\mathbb E\, |\mathcal N_\C(0,\sigma^2)|^2=\sigma^2.
\]
\begin{thm}[Convergence in law for the rescaled coefficients]\label{thm:CLT}
	The following convergence in distribution holds:
	\begin{equation*}
		n^{\frac{1-\beta^2}{2}}\, c_n \longrightarrow  \mathcal N_{\mathbb C}(0,\kappa(\beta)),\qquad \text{ as } n\rightarrow\infty,
	\end{equation*}
	where $\kappa(\beta) = \frac{1}{\pi} \Gamma(1-\beta^2)\sin(\frac{\pi\beta^2}{2})$. Moreover, for any fixed integer $k\geq0$, we have the finite-dimensional convergence
	\begin{equation*}
		n^{\frac{1-\beta^2}{2}}\, (c_n,\dots,c_{n+k}) \longrightarrow  (\mathcal N^0_{\mathbb C}(0,\kappa(\beta)),\dots,\mathcal N^k_{\mathbb C}(0,\kappa(\beta))),\qquad \text{ as } n\rightarrow\infty,
	\end{equation*}
	where the Gaussian random variables are independent.
\end{thm}
Let $W$ be a complex white noise of intensity $\kappa(\beta)$, i.e., the random distribution given by
\begin{equation*}
	W \de \sum_{k\in\Z}\xi_k\,\e^{\ii k\theta},
\end{equation*}
where $(\xi_k)$ is an i.i.d. sequence of complex Gaussian random variables $\mathcal N_{\mathbb C}(0,\kappa(\beta))$. A similar noise-like behavior was established for the two-dimensional Gaussian free field in \cite{arubaverezjegojunnila25}.
\begin{thm}[Convergence toward a complex white noise]\label{thm:whitenoise} 
	Let $s<-\frac{1}{2}$. We have the following convergence in distribution in $H^{s}(\mathbb T)$:
	\begin{equation*}
		n^{\frac{1-\beta^2}{2}}\,\e^{\ii n\theta} \,\mathrm M_{\ii\beta}\longrightarrow W,\qquad \text{ as } n\rightarrow\infty.
	\end{equation*}
\end{thm}

\begin{rem}
	Theorems \ref{thm:CLT} and \ref{thm:whitenoise} are stated only for the exactly log-correlated field. Their proofs rely on Coulomb-gas and Jack-polynomial identities, which are no longer available under perturbation. A formal computation based on the factorization $\M^g=e^{\beta^2 g(\theta,\theta)/2} e^{\ii\beta Y(\theta)}\cdot \M$ from Lemma \ref{lem:factor} suggests that $n^{(1-\beta^2)/2}\,e^{\ii n\theta}\,\M^g$ should converge to a complex Gaussian white noise with non-homogeneous intensity $\kappa(\beta)\,e^{\beta^2 g(\theta,\theta)}$. This can be shown in the case where $g$ is positive semidefinite, by taking $Y$ independent of $X$ in the factorization. A rigorous proof in the general case, which we do not pursue here, would require a careful treatment of the asymptotic dependence between $Y$ and the high-frequency content of $\M$.
\end{rem}

Beyond the purely imaginary case, the Sobolev regularity of complex multiplicative chaos obtained in \cite{junnilasaksmanviitasaari2019} suggests that the imaginary part of the parameter should lower the Fourier dimension by $\beta^2$. This suggests the following open problem.

\begin{conj}\label{conj:1}
	For $\gamma=\alpha+\ii\beta\in \mathcal D$, one has almost surely
	\begin{equation*}
		\mathrm{dim}_{\mathrm F}(\mathrm M_{\gamma}) = \mathrm{dim}_{\mathrm F}(\mathrm M_{\alpha}) -\beta^2.
	\end{equation*}
\end{conj}

\subsection*{Strategy.}
We proceed by the method of moments. The rotational invariance of the field implies that each Fourier mode $c_n$ is isotropic in the complex plane. Hence, for a single mode, convergence in law to a complex Gaussian is reduced to the convergence of the absolute moments $\mathbb E |c_n|^{2N}$. These moments have a Coulomb-gas integral representation on the circle. Using Stanley's Cauchy identity, together with the shift property and orthogonality of Jack polynomials, one reduces $\mathbb E |c_n|^{2N}$ to an explicit positive sum over partitions. The asymptotic analysis of this sum gives the limiting Gaussian moments and, through a summability argument, the lower bound on the Fourier dimension. From there, using concentration of $c_n$ and Kolmogorov's 0-1 law, we settle the critical Sobolev regularity of $\M$. 

The case of a general chaos $\M^g$ is treated by combining a spectral coupling with a convolution argument. In the stationary case, a Fourier decomposition realizes $X^g$ as the sum of a standard log-correlated field $X$ and a Gaussian remainder, whose regularity is dictated by the decay of $\hat g$. The pointwise decay of $c_n$ obtained in Theorem \ref{thm:1} then transfers to the Fourier coefficients of $\M^g$ through a convolution argument. The non-stationary case follows a similar scheme, with a Karhunen–Loève type coupling obtained in \cite{JSW20}.

The joint convergence of several modes requires a refinement of the single-mode argument. Mixed moments of $c_n,\ldots,c_{n+k}$ lead, after symmetrization, to products of Jack polynomials by certain symmetric polynomials. These products are controlled by Pieri formulas, which produce a finite decomposition indexed by shapes added to the underlying partitions. The main idea is then to show that it is sufficient to analyze this decomposition in a large-gap regime, where the gaps between consecutive parts of the partition tend to infinity. In this regime, the Pieri coefficients and other relevant quantities are asymptotically equal to $1$. This enables to reveal the leading contribution of the asymptotics of the mixed moments.

\subsection*{Structure of the article.}
In Section \ref{sec:mom1}, we study the precise asymptotic behavior of the moments of a single Fourier mode. This is the core of the proof of the convergence of one rescaled coefficient and, together with a standard summability argument, yields the optimal lower bound on the Fourier dimension. In particular, this section proves Theorem \ref{thm:1} and the first part of Theorem \ref{thm:CLT}. The section ends with the proof of Corollary \ref{cor:sobcritic}. Section \ref{sec:robustness} is dedicated to the proof of the generalization to general chaos. Section \ref{sec:mom2} is devoted to the asymptotic analysis of joint moments of several Fourier modes. This allows us to prove the joint convergence statement in the second part of Theorem \ref{thm:CLT}. Finally, in Section \ref{sec:whitenoise}, we derive the convergence of the shifted chaos toward complex white noise, namely Theorem \ref{thm:whitenoise}. Appendix \ref{partition} collects the elementary facts on partitions and Young diagrams used throughout the paper.

\section{The moments of $c_n$}\label{sec:mom1}
For complex-valued random variables, the convergence in distribution is a consequence of the convergence of the mixed moments
\begin{equation*}
	\E{Z_n^p \overline{Z_n}^q} \longrightarrow \E{Z^p \overline{Z}^q} \qquad \text{ for }p,q\in \N,
\end{equation*}
as soon as the law of $Z$ is characterized by its moments. This is the case for the complex Gaussian distribution $\mathcal N_{\mathbb C}$. Here, an additional symmetry greatly simplifies matters, since the law of $c_n$
is isotropic for all $n>0$. Recall that a complex-valued random variable $Z$  is isotropic if for every $\alpha\in\R$, $Z$ has the same law as $\e^{\ii\alpha} Z$. It is equivalent to the following two properties:
\begin{enumerate}
	\item $\theta$ is uniformly distributed on $\left[0,2\pi\right]$,
	\item $\theta$ is independent of $R$.
\end{enumerate}
where $(R,\theta)$ is the polar decomposition of $Z$.
For isotropic random variables, the convergence in distribution is equivalent to the convergence in distribution of the modulus $R$. Indeed, suppose that $f$ is continuous and bounded on $\mathbb C$ and let $g(r)=\frac{1}{2\pi}\int_{|z|=r}f(z)\dz$, then
\begin{equation*}
	\E{f(Z_n)} = \E{f(R_n e^{\ii \theta_n})} = \E{\Ec{f(R_n e^{\ii \theta_n})}{R_n}} = \E{g(R_n)}\to \E{g(R)} = \E{f(Z)}.
\end{equation*}
The isotropy of $c_n$ is a consequence of the invariance by rotation of the field $X$:
\begin{equation*}
	e^{\ii \alpha} c_n 
	=
	\frac{1}{2\pi}\int_0^{2\pi} \e^{-\ii n \theta} \,\e^{\ii\beta X(\theta+\frac{\alpha}{n})+\frac{\beta^2}{2}\E{X(\theta+\frac{\alpha}{n})^2}}\,\dth \overset{\mathrm{(d)}}{=} c_n.
\end{equation*}
Therefore, to prove convergence in law of the rescaled Fourier coefficient to $\mathcal N_{\mathbb C}(0,\kappa(\beta))$,
it suffices to show that, for each $N$,
\begin{equation*}
	n^{(1-\beta^2)N}\mathbb E|c_n|^{2N}\longrightarrow N!\,\kappa(\beta)^N, \quad \text{as }n\to\infty.
\end{equation*}

On the circle, the $2N$-th moment of $c_n$ is explicitly given by the Coulomb-gas type integral
\begin{equation}\label{eq:2Nmo}
	\mathbb E|c_n|^{2N}
	=
	\frac{1}{(2\pi)^{2N}}
	\int_{[0,2\pi]^{2N}}
	\e^{-\ii n\sum_{i=1}^N(\theta_i-\theta_i')}\;
	\dfrac{\prod_{1\leq i<j\leq N} |\e^{\ii\theta_i}-\e^{\ii\theta_j}|^{\beta^2}\,|\e^{\ii\theta'_i}-\e^{\ii\theta'_j}|^{\beta^2}}{
	\prod_{1\leq i,j\leq N}|\e^{\ii\theta_i}- \e^{\ii\theta'_j}|^{\beta^2}}\,
	\dth\,\dthp.
\end{equation}
The factors $\prod_{i<j}|\e^{\ii\theta_i}-\e^{\ii\theta_j}|^{\beta^2}$ induce a natural scalar product, which we study in the next section.
\begin{rem}
	When $n=0$, the above integral coincides with the partition function of the
	two-component log-gas on the unit circle.
	Equivalently, it may be viewed as a two-dimensional Coulomb gas with charges
	constrained to lie on the circle. This circle model has appeared in the physics
	literature, see e.g. \cite{fendleylesagesaleur95}. We also mention that related
	two-component Coulomb gases in the plane have been extensively studied in
	\cite{lebleserfatyzeitouni17} and subsequent works. There, the asymptotic behavior is studied in the different regime $N\rightarrow\infty$.
\end{rem}

\subsection{Selberg inner product and Jack polynomials}
From now on, let us denote $\gamma=\beta^2/2$\footnote{This choice of
	notation is made to match the standard parameter used for Jack polynomials.
	It should not be confused with the intermittency parameter of GMC, which in
	the imaginary case considered here is $\ii\beta$.}. On $\mathbb{T}^N$, the Selberg inner product is defined as
\begin{equation*}
	\langle f,g\rangle_{\gamma}\;=\;\frac{1}{(2\pi)^N}\int_{[0,2\pi]^{N}}
	f(\e^{\ii\theta})\; \overline{g(\e^{\ii\theta})}\,
	\prod_{1\le i<j\le N}|\e^{\ii\theta_i}-\e^{\ii\theta_j}|^{2\gamma}\,\dth,
\end{equation*}
where $f(\e^{\ii\theta}) = f(\e^{\ii\theta_1},\cdots,\e^{\ii\theta_N})$.
There is a natural orthogonal basis of symmetric and homogeneous polynomials with real coefficients $P_{\lambda}^{(1/\gamma)}$, called the Jack polynomials, which are indexed by integer partitions $\lambda=\left(\lambda_1\ge\cdots\ge\lambda_{N}\ge0\right)$ of length smaller than $N$. The definitions and properties of partitions used throughout the paper are gathered in Appendix \ref{partition}. The orthogonality of the polynomials reads
\begin{equation}\label{eq:orth}
	\int_{[0,2\pi]^{N}}
	P_{\lambda}^{(1/\gamma)}(\e^{\ii\theta})\,
	P_{\mu}^{(1/\gamma)}(\e^{-\ii\theta})\,
	\prod_{i<j}|\e^{\ii\theta_i}-\e^{\ii\theta_j}|^{2\gamma}\,\dth
	=
	(2\pi)^{N}\,\delta_{\lambda\mu}\,\left\|P_{\lambda}^{(1/\gamma)}\right\|_{\gamma}^{2}.
\end{equation}
For an extensive review of Jack polynomials, we encourage the reader to consult Macdonald's book \cite[Chapter VI]{Macdonald95} and for its connection with the Selberg integral, we refer to \cite{ForresterWarnaar08}.

We gather here the principal facts about Jack polynomials required for the present work. To lighten the notation, we omit  the superscript $(1/\ga)$ in what follows. First, their norm is given by
\begin{equation}\label{eq:norm}
	\left\|P_\lambda\right\|_{\gamma}^{2}
	=
	\frac{\Gamma(1+N\gamma)}{\Gamma(1+\gamma)^{N}}\;
	\frac{c'_\lambda}{[1+(N-1)\gamma]_{\lambda}}\,P_\lambda(1^N).
\end{equation}
The generalized Pochhammer symbol is defined by
\begin{equation*}
	[b]_\lambda=\displaystyle\prod_{i=1}^{N}\left(b+(1-i)\gamma\right)_{\lambda_i}, \qquad
	(b)_k = b(b+1)\cdots(b+k-1),
\end{equation*}
and
\begin{equation}
	c_\lambda=\prod_{s\in\lambda}\left(a_\lambda(s)+\gamma\,l_\lambda(s)+\gamma\right),\qquad
	c'_\lambda=\prod_{s\in\lambda}\left(a_\lambda(s)+\gamma\,l_\lambda(s)+1\right),
\end{equation}
where the products in $c_\lambda$ and $c'_\lambda$ are taken over boxes $s$ of the Young diagram of $\lambda$ and involve the arm- and leg-lengths $a_\lambda(s)$ and $l_\lambda(s)$ (see Appendix \ref{partition} for a precise definition). The value at $1=(1,\cdots,1)$ of the polynomials is given by
\begin{equation}\label{eq:value1}
	P_\lambda(1) = \frac{\left[N\gamma\right]_{\lambda}}{c_{\lambda}}.
\end{equation}
\subsection*{Stanley's Cauchy identity.} The following identity will play a key role in what follows:
\begin{equation}\label{eq:CS}
	\prod_{i,j}(1-x_i y_j)^{-\gamma}
	=
	\sum_{\lambda}\frac{c_\lambda(\gamma)}{c'_\lambda(\gamma)}\,
	P_\lambda(x)\,P_\lambda(y).
\end{equation}

\subsection*{Shift property.} One last property is 
\begin{equation}\label{eq:shift}
	(x_1\cdots x_N)^n\,P_\lambda(x)=P_{\lambda+n}(x),
\end{equation}
where $\lambda+n$ is obtained by adding $n$ to each part of $\lambda$, in other words $\lambda+n = (\lambda_1+n,\dots,\lambda_N+n)$.
As a consequence,
\begin{equation*}
	P_\lambda(1) =P_{\lambda+n}(1) = \frac{\left[N\gamma\right]_{\lambda+n}}{c_{\lambda+n}}.
\end{equation*}

\subsection{Jack expansion of the moments}\label{sub:jackexp}
Applying \eqref{eq:CS} to $x_i=\e^{\ii\theta_i}$ and $y_j=\e^{-\ii\theta'_j}$, and using the fact that Jack coefficients are real, hence $\overline{P_\lambda(\e^{\ii\theta})}=P_\lambda(\e^{-\ii\theta})$ on the unit circle, we may formally rewrite the denominator in the integral representation as 
\begin{align}\label{eq:2sum}
	\prod_{i,j}|\e^{\ii\theta_i}- \e^{\ii\theta'_j}|^{-2\gamma}
	&=\prod_{i,j}(1-\e^{\ii\theta_i}\e^{-\ii\theta'_j})^{-\gamma}\,
	\prod_{i,j}(1-\e^{-\ii\theta_i}\e^{\ii\theta_j'})^{-\gamma}\\
	&=\sum_{\lambda,\nu}\frac{c_\lambda}{c'_\lambda}\,\frac{c_\nu}{c'_\nu}
	P_\lambda(\e^{\ii \theta})\,P_\lambda(\e^{-\ii \theta'})P_\nu(\e^{-\ii \theta})\,P_\nu(\e^{\ii \theta'}).\nonumber
\end{align}
Combining \eqref{eq:2Nmo} with \eqref{eq:2sum} yields the double series expansion
\begin{align}\label{eq:doublesum}
	\mathbb E|c_n|^{2N}
	&=\frac{1}{(2\pi)^{2N}}\int_{[0,2\pi]^{2N}}
	\e^{-\ii n\sum_i(\theta_i-\theta_i')}
	\prod_{ i<j} |\e^{\ii\theta_i}-\e^{\ii\theta_j}|^{2\gamma}\,|\e^{\ii\theta'_i}-\e^{\ii\theta'_j}|^{2\gamma}\,
	\prod_{i,j}|\e^{\ii\theta_i}- \e^{\ii\theta'_j}|^{-2\gamma}\,
	\dth\,\dthp\nonumber\\
	&=\frac{1}{(2\pi)^{2N}}\int_{[0,2\pi]^{2N}}
	\e^{\ii n\sum_i(\theta_i-\theta_i')}
	\prod_{ i<j} |\e^{\ii\theta_i}-\e^{\ii\theta_j}|^{2\gamma}\,|\e^{\ii\theta'_i}-\e^{\ii\theta'_j}|^{2\gamma}\,
	\prod_{i,j}|\e^{\ii\theta_i}- \e^{\ii\theta'_j}|^{-2\gamma}\,
	\dth\,\dthp\nonumber\\
	&=\sum_{\lambda,\nu}\frac{c_\lambda}{c'_\lambda}\frac{c_\nu}{c'_\nu}\,
	\frac{1}{(2\pi)^{2N}}\int_{[0,2\pi]^{2N}}
	\e^{\ii n\sum_i(\theta_i-\theta_i')}\,P_\lambda(\e^{\ii\theta})\,P_\lambda(\e^{-\ii\theta'})\;P_\nu(\e^{-\ii\theta})\;P_\nu(\e^{\ii\theta'})\\
	&\qquad\qquad\qquad\qquad\qquad\prod_{ i<j} |\e^{\ii\theta_i}-\e^{\ii\theta_j}|^{2\gamma}\,|\e^{\ii\theta'_i}-\e^{\ii\theta'_j}|^{2\gamma}\,\dth\,\dthp\nonumber.
\end{align}
A rigorous justification is given at the end of this subsection.

Now, the shift property \eqref{eq:shift} together with the orthogonality \eqref{eq:orth} reduces the expression above to the following series:
\begin{align*}
	\mathbb E|c_n|^{2N}
	&=\sum_{\lambda,\nu}\frac{c_\lambda}{c'_\lambda}\frac{c_\nu}{c'_\nu}\,
	\frac{1}{(2\pi)^{2N}} \int P_{\lambda+n}(\e^{\ii\theta})\, P_{\lambda+n}(\e^{-\ii\theta'})\,P_\nu(\e^{-\ii\theta})\,P_\nu(\e^{\ii\theta'})\prod_{ i<j} |\e^{\ii\theta_i}-\e^{\ii\theta_j}|^{2\gamma}\,|\e^{\ii\theta'_i}-\e^{\ii\theta'_j}|^{2\gamma}\,\,\dth\,\dthp\nonumber\\
	&=\sum_{\lambda,\nu}\frac{c_\lambda}{c'_\lambda}\frac{c_\nu}{c'_\nu}\,
	\Bigg|\frac{1}{(2\pi)^{N}}\int_{[0,2\pi]^N} P_{\lambda+n}(\e^{\ii\theta})\, P_\nu(\e^{-\ii\theta})
	\prod_{i<j} |\e^{\ii\theta_i}-\e^{\ii\theta_j}|^{2\gamma}\,\dth\Bigg|^2\\
	&=\sum_{\lambda}\frac{c_\lambda}{c'_\lambda}\frac{c_{\lambda+n}}{c'_{\lambda+n}}\,
	 \nor{P_{\lambda+n}}^4.
\end{align*}
The norm given in \eqref{eq:norm} and the value at $1$ given in \eqref{eq:value1} yield
\begin{align}
\mathbb E|c_n|^{2N}
	&=\frac{\Gamma(1+N\gamma)^2}{\Gamma(1+\gamma)^{2N}}\sum_{\lambda}\frac{c_\lambda}{c'_\lambda}\frac{c_{\lambda+n}}{c'_{\lambda+n}} \frac{c^{'2}_{\lambda+n}}{[1+(N-1)\gamma]_{\lambda+n}^2}\frac{\left[N\gamma\right]_{\lambda}}{c_{\lambda}} \frac{\left[N\gamma\right]_{\lambda+n}}{c_{\lambda+n}}\nonumber\\	&=\frac{\Gamma(1+N\gamma)^2}{\Gamma(1+\gamma)^{2N}}\sum_{\lambda}\frac{c'_{\lambda+n}}{c'_{\lambda}}
	\frac{\left[N\gamma\right]_\lambda\left[N\gamma\right]_{\lambda+n}}{\left[1+(N-1)\gamma\right]_{\lambda+n}^2}.\label{eq:serie1}
\end{align}
\subsection*{Adding a rectangle.} The Young diagram of $\lambda+n$ is obtained from that of $\lambda$ by attaching a $N\times n$ rectangle on the left. This yields the explicit product formulas:
\begin{equation}\label{eq:rectc}
	c_{\lambda+n}=c_\lambda\,\prod_{i=1}^{N}\frac{\Gamma(\gamma(N-i+1)+\lambda_i+n)}{\Gamma(\gamma(N-i+1)+\lambda_i)},
\end{equation}
\begin{equation}\label{eq:rectc'}
	c'_{\lambda+n}=c'_\lambda\,\prod_{i=1}^{N}\frac{\Gamma(\gamma(N-i+1)+\lambda_i+n+1-\gamma)}{\Gamma(\gamma(N-i+1)+\lambda_i+1-\gamma)}.
\end{equation}
Moreover, the generalized Pochhammer symbol admits a representation in terms of $\Gamma$ functions
\begin{equation*}
	[b]^{(\gamma)}_{\lambda}
	=
	\prod_{i=1}^{N}\frac{\Gamma\left(b-(i-1)\gamma+\lambda_i\right)}{\Gamma\left(b-(i-1)\gamma\right)}.
\end{equation*}
Thus the second term in the series (\ref{eq:serie1}) is equal to
\begin{align}\label{eq:2ndterm}
	\frac{[N\gamma]_{\lambda}\,[N\gamma]_{\lambda+n}}{[1+(N-1)\gamma]^{2}_{\lambda+n}}
	&=
	\prod_{i=1}^{N}\frac{\Gamma((N-i+1)\gamma+\lambda_i)}{\Gamma((N-i+1)\gamma)}\;
	\prod_{i=1}^{N}\frac{\Gamma((N-i+1)\gamma+\lambda_i+n)}{\Gamma((N-i+1)\gamma)}\\
	&\;\times\prod_{i=1}^{N}\left(\frac{\Gamma(1+(N-i)\gamma+\lambda_i+n)}{\Gamma(1+(N-i)\gamma)}\right)^{-2}.\nonumber
\end{align}

Using \eqref{eq:rectc'} and \eqref{eq:2ndterm}, one gets
\begin{align*}
	\mathbb E|c_n|^{2N}
	&=\frac{\Gamma(1+N\gamma)^2}{\Gamma(1+\gamma)^{2N}}\sum_{\lambda}\frac{c'_{\lambda+n}}{c'_{\lambda}}\frac{\left[N\gamma\right]_\lambda\left[N\gamma\right]_{\lambda+n}}{\left[1+(N-1)\gamma\right]_{\lambda+n}^2}\\
	&=\frac{\Gamma(1+N\gamma)^2}{\Gamma(1+\gamma)^{2N}}\left(\frac{(N-1)!\, \gamma^{N-1}}{\Gamma(N\gamma)}\right)^2\\
	&\qquad \sum_{\lambda}\; \prod_{i=1}^{N}
	\frac{\G{(N-i+1)\gamma+\lambda_i}}
	{\G{(N-i+1)\gamma+\lambda_i+1-\gamma}}
	\;\prod_{i=1}^{N}
	\frac{\G{(N-i+1)\gamma+\lambda_i+n}}
	{\G{(N-i+1)\gamma+\lambda_i+n+1-\gamma}}\\
	&= \frac{1}{\Gamma(\gamma)^{2N}} \,(N!)^2\\
	&\qquad \sum_{\lambda}\; \prod_{i=1}^{N}
	\frac{\G{(N-i+1)\gamma+\lambda_i}}
	{\G{(N-i+1)\gamma+\lambda_i+1-\gamma}}
	\;\prod_{i=1}^{N}
	\frac{\G{(N-i+1)\gamma+\lambda_i+n}}
	{\G{(N-i+1)\gamma+\lambda_i+n+1-\gamma}},
\end{align*}
where we used $\Gamma(z+1) = z\Gamma(z)$ in the second line.

The end of this subsection is dedicated to the proof of \eqref{eq:doublesum}.
\begin{proof}
Let $x_i=r\e^{\ii \theta_i}$ and $y_j=r\e^{-\ii \theta_i'}$ with $r<1$. The absolute convergence of Stanley's Cauchy expansion in this case leads to
	\begin{align*}
		I_{r}(N,n)
		&\de \int_{[0,2\pi]^{2N}} \e^{\ii n\sum_i(\theta_i-\theta_i')}
		\prod_{i<j} |\e^{\ii\theta_i}-\e^{\ii\theta_j}|^{2\gamma}\,|\e^{\ii\theta'_i}-\e^{\ii\theta'_j}|^{2\gamma}\,
		\prod_{i,j}|1-r^2\e^{\ii\theta_i}\e^{-\ii\theta'_j}|^{-2\gamma}\,
		\dth\,\dthp\\
		&=\sum_{\lambda,\nu}\frac{c_\lambda}{c'_\lambda}\frac{c_\nu}{c'_\nu}
		\int_{[0,2\pi]^{2N}}
		\e^{\ii n\sum_i(\theta_i-\theta_i')}\,P_\lambda(r\e^{\ii\theta})\,P_\lambda(r\e^{-\ii\theta'})\;P_\nu(r\e^{-\ii\theta})\;P_\nu(r\e^{\ii\theta'})\\
		&\hspace{7cm}\times\prod_{ i<j} |\e^{\ii\theta_i}-\e^{\ii\theta_j}|^{2\gamma}\,|\e^{\ii\theta'_i}-\e^{\ii\theta'_j}|^{2\gamma}\,\dth\,\dthp.
	\end{align*}
	Now, since $P_\lambda$ is homogeneous of degree $|\lambda|$,
	\begin{align*}
		I_r(N,n)
		&=\sum_{\lambda,\nu}\frac{c_\lambda}{c'_\lambda}\frac{c_\nu}{c'_\nu} r^{2(|\lambda|+|\nu|)}
		\int_{[0,2\pi]^{2N}}
		\e^{\ii n\sum_i(\theta_i-\theta_i')}\,P_\lambda(\e^{\ii\theta})\,P_\lambda(\e^{-\ii\theta'})\;P_\nu(\e^{-\ii\theta})\;P_\nu(\e^{\ii\theta'})\\
		&\hspace{7cm}\times\prod_{ i<j} |\e^{\ii\theta_i}-\e^{\ii\theta_j}|^{2\gamma}\,|\e^{\ii\theta'_i}-\e^{\ii\theta'_j}|^{2\gamma}\,\dth\,\dthp\\
		&=\sum_{\lambda,\nu}\frac{c_\lambda}{c'_\lambda}\frac{c_\nu}{c'_\nu}\,r^{2(|\lambda|+|\nu|)}\,
		\Bigg|\int_{[0,2\pi]^N} P_{\lambda+n}(\e^{\ii\theta})\, P_\nu(\e^{-\ii\theta})
		\prod_{i<j} |\e^{\ii\theta_i}-\e^{\ii\theta_j}|^{2\gamma}\,\dth\Bigg|^2\\
		&=\sum_{\lambda}\frac{c_\lambda}{c'_\lambda}\frac{c_{\lambda+n}}{c'_{\lambda+n}}\,r^{2(|\lambda|+|\lambda+n|)}\,
		(2\pi)^{2N} \nor{P_{\lambda+n}}^4.
	\end{align*}
and then pass to the limit $r\uparrow1$, using dominated convergence for the integral and monotone convergence for the double series.
\end{proof}

\subsection{Asymptotics of the moments}
It will be more convenient to write the above summand by introducing
\begin{equation*}
	F(x)= \frac{\Gamma(x)}{\Gamma(x+1-\gamma)}, \quad
	G_\lambda(n) = \prod_{i=1}^{N} F\left((N-i+1)\gamma+\lambda_i\right) F\left((N-i+1)\gamma+\lambda_i+n\right),\quad
	S(n) = \sum_{\lambda} G_\lambda(n).
\end{equation*}
This yields the following expression for the moment:
\begin{equation*}\label{eqref:cn}
	\mathbb E|c_n|^{2N}=\frac{1}{\Gamma(\gamma)^{2N}}(N!)^2\,S(n).
\end{equation*}

\begin{prop}\label{pro:S}
	As $n$ tends to infinity,
	\begin{equation*}
		S(n)
		\sim
		\frac{1}{N!}\left(\int_{0}^{\infty}\frac{\dx}{x^{1-\gamma}\left(1+x\right)^{1-\gamma}}\right)^{N}\,
		\frac{1}{n^{N(1-2\gamma)}}.
	\end{equation*}
\end{prop}

\begin{proof}

\emph{Lower bound.}
Using Wendel's inequality $F(x)>x^{\ga-1}$ valid for all $x>0$, one gets the lower bound
\begin{align*}
	S(n)
	&\ge \sum_{\lambda}\prod_{i=1}^{N}
	\frac{1}{\left(\left(N-i+1\right)\gamma+\lambda_i\right)^{1-\gamma}
	\left(\left(N-i+1\right)\gamma+\lambda_i+n\right)^{1-\gamma}}\\
	&\ge \sum_{\lambda_1\ge\cdots\ge\lambda_N\ge 0}
	\prod_{i=1}^{N}
	\frac{1}{\left(N\gamma+\lambda_i\right)^{1-\gamma}
	\left(N\gamma+\lambda_i+n\right)^{1-\gamma}}\\
	&\ge \frac{1}{N!}\sum_{\lambda_1,\ldots,\lambda_N\ge 0}
	\prod_{i=1}^{N}
	\frac{1}{\left(N\gamma+\lambda_i\right)^{1-\gamma}
	\left(N\gamma+\lambda_i+n\right)^{1-\gamma}}\\
	&= \frac{1}{N!}\;\left(\sum_{k=0}^{\infty}
	\frac{1}{\left(N\gamma+k\right)^{1-\gamma}\left(N\gamma+k+n\right)^{1-\gamma}}\right)^{N}.
\end{align*}
To analyze the series above, it is convenient to factor out the scaling in $n$ and then recognize a Riemann sum:
\begin{equation*}
	\sum_{k=0}^{\infty}\frac{1}{\left(N\gamma+k\right)^{1-\gamma}\left(N\gamma+k+n\right)^{1-\gamma}}
	= 
	\frac{1}{n^{1-2\gamma}}\;\frac{1}{n}\sum_{k=0}^{\infty}
	\frac{1}{\left(\frac{N\gamma+k}{n}\right)^{1-\gamma}\left(1+\frac{N\gamma+k}{n}\right)^{1-\gamma}}.
\end{equation*}
Since the function $x\mapsto x^{-(1-\gamma)}\left(1+x\right)^{-(1-\gamma)}$ is decreasing and integrable on $(0,\infty)$, a series-integral comparison yields
\begin{equation*}
	\sum_{k=0}^{\infty}\frac{1}{\left(N\gamma+k\right)^{1-\gamma}\left(N\gamma+k+n\right)^{1-\gamma}}
	\sim \frac{1}{n^{1-2\gamma}}\;\int_{0}^{\infty}\frac{\dx}{x^{1-\gamma}\left(1+x\right)^{1-\gamma}},\qquad\text{as }n\to\infty.
\end{equation*}
In particular, we obtain the lower bound
\begin{equation}\label{eq:liminf}
	\liminf_{n\to\infty} n^{N(1-2\gamma)}\,S(n)
	\ge
	\frac{1}{N!}\;\left(\int_{0}^{\infty}\frac{\dx}{x^{1-\gamma}\left(1+x\right)^{1-\gamma}}\right)^{N}.
\end{equation}

\emph{Upper bound.}
The upper bound needs a bit more care. Split the sum the following way
\begin{equation*}
	S(n)\;=\;S_{1}^\ep(n)\;+\;S_{2}^\ep(n)\;+\;S_{3}^\ep(n),
\end{equation*}
where $S_{1}^\ep(n)$ sums over strictly decreasing partitions with $\lambda_1>\cdots>\lambda_N>\ep n$, $S_{2}^\ep(n)$ collects those with at least one equality $\lambda_i=\lambda_{i+1}$ and $\lambda_N>\ep n$, and $S_3^\ep(n)$ sums over $\lambda$ such that $\lambda_N\leq \ep n$. We prove that $S_2^\ep(n)$ and $S_3^\ep(n)$ are negligible and $S_1^\ep(n)$ produces the correct leading order.

For the second term, note that $\lambda_{i}=\lambda_j$ implies that all the $\lambda_k$ between $i$ and $j$ are equal. If $\Lambda_i = \{\lambda \text{ such that } \lambda_i = \lambda_{i+1}\}$, one gets the following upper bound 
\begin{equation*}
	S_2^\ep(n) \leq \sum_{i=1}^{N-1} \sum_{\lambda\in\Lambda_i} G_{\lambda}(n).
\end{equation*}
The terms can be controlled using the elementary bound $F(x)\le C\,x^{\gamma-1}$, valid for $x\ge \gamma$:
\begin{align*}
	\sum_{\lambda\in\Lambda_i} G_{\lambda}(n)
	&\le C^{2N}\sum_{\lambda_1,\ldots,\lambda_i,\lambda_{i+2},\ldots,\lambda_N\ge 0}
	\left(\prod_{k\neq i,i+1}\frac{1}{\left(\gamma+\lambda_k\right)^{1-\gamma}\left(\gamma+\lambda_k+n\right)^{1-\gamma}}\right)
	\frac{1}{\left(\gamma+\lambda_i\right)^{2-2\gamma}\left(\gamma+\lambda_i+n\right)^{2-2\gamma}}\\
	&\le C^{2N}\frac{1}{n^{(N-2)(1-2\gamma)}}\left(\frac{1}{n}\sum_{k=0}^{\infty}
	\frac{1}{\left(\frac{\gamma+k}{n}\right)^{1-\gamma}\left(1+\frac{\gamma+k}{n}\right)^{1-\gamma}}\right)^{N-2}
	\sum_{k=0}^{\infty}\frac{1}{\left(\gamma+k\right)^{2-2\gamma}\left(\gamma+k+n\right)^{2-2\gamma}}.
\end{align*}
Since $\gamma<\frac{1}{2}$, the last series is 
\begin{equation*}
	\sum_{k=0}^{\infty}\frac{1}{\left(\gamma+k\right)^{2-2\gamma}\left(\gamma+k+n\right)^{2-2\gamma}} 
	\leq 
	\frac{1}{n^{2-2\ga}}\sum_{k=0}^{\infty}\frac{1}{\left(\gamma+k\right)^{2-2\gamma}}.
\end{equation*}
Hence the whole contribution is $o\left(n^{-N(1-2\gamma)}\right)$ as $n\to\infty$ and
summing over $i=1,\ldots,N-1$ yields $S_{2}^\ep(n)=o\left(n^{-N(1-2\gamma)}\right)$.
Using the same upper bound, one gets
\begin{align*}
	S_3^\varepsilon(n)
	&\le C^{2N}\sum_{\substack{\lambda_1\ge\cdots\ge \lambda_N\ge 0\\ \lambda_N<\varepsilon n}}
	\prod_{i=1}^{N}
	\frac{1}{\left(\gamma+\lambda_i\right)^{1-\gamma}\left(\gamma+\lambda_i+n\right)^{1-\gamma}}\\
	&\le C^{2N}\,\frac{1}{n^{N(1-2\gamma)}}\left(\frac{1}{n}\sum_{k=0}^{\infty}
	\frac{1}{\left(\frac{\gamma+k}{n}\right)^{1-\gamma}\left(1+\frac{\gamma+k}{n}\right)^{1-\gamma}}\right)^{N-1}
	\left(\frac{1}{n}\sum_{k=0}^{\lfloor \varepsilon n\rfloor}
	\frac{1}{\left(\frac{\gamma+k}{n}\right)^{1-\gamma}\left(1+\frac{\gamma+k}{n}\right)^{1-\gamma}}\right).\\
\end{align*}
Passing to the limit $n\to\infty$, the last factor converges to
$\int_{0}^{\varepsilon} x^{-(1-\gamma)}\left(1+x\right)^{-(1-\gamma)}\,\dx$, which vanishes as $\varepsilon\to0$.
Hence,
\begin{equation*}
	\lim\limits_{\ep\to0}\limsup\limits_{n\to\infty} n^{N(1-2\ga)} S_3^\ep(n)=0.
\end{equation*}

Turning to the main term $S_{1}^\ep(n)$, we use the existence of a constant $B(A)$ such that $F(x)\leq \frac{B(A)}{x^{1-\ga}}$ for $x>A$, with $B(A)\to 1$ as $A\to \infty$. Thus
\begin{align*}
	S_1^\ep(n)
	&\le B(\ep n)^{2N}\sum_{\lambda_1>\ldots>\lambda_N>\ep n}
	\prod_{i=1}^{N}
	\frac{1}{\left(\gamma+\lambda_i\right)^{1-\gamma}\left(\gamma+\lambda_i+n\right)^{1-\gamma}}\\
	&\le \frac{B(\ep n)^{2N}}{N!}\sum_{\lambda_1,\ldots,\lambda_N\geq 0}
	\prod_{i=1}^{N}
	\frac{1}{\left(\gamma+\lambda_i\right)^{1-\gamma}\left(\gamma+\lambda_i+n\right)^{1-\gamma}}\\
	&\le \frac{B(\ep n)^{2N}}{N!}\,\frac{1}{n^{N(1-2\gamma)}}\left(\frac{1}{n}\sum_{k=0}^{\infty}
	\frac{1}{\left(\frac{\gamma+k}{n}\right)^{1-\gamma}\left(1+\frac{\gamma+k}{n}\right)^{1-\gamma}}\right)^{N},
\end{align*}
which, as $n\to\infty$, is equivalent to
$\frac{1}{N!}\,n^{-N(1-2\gamma)}\left(\int_{0}^{\infty}x^{-(1-\gamma)}\left(1+x\right)^{-(1-\gamma)}\,\dx\right)^{N}$.
Combining the three pieces with \eqref{eq:liminf}, we find
\begin{equation}\label{eq:eq}
	\lim_{n\to\infty} n^{N(1-2\gamma)}\,S(n)
	=
	\frac{1}{N!}\left(\int_{0}^{\infty}\frac{\dx}{x^{1-\gamma}\left(1+x\right)^{1-\gamma}}\right)^{N}.
\end{equation}
\end{proof}

Applying the change of variables $t=\dfrac{x}{1+x}$ to \eqref{eq:eq} yields
\begin{equation*}
	\int_{0}^{\infty}\frac{\dx}{x^{1-\gamma}\left(1+x\right)^{1-\gamma}}
	=
	\int_{0}^{1} t^{\gamma-1}\left(1-t\right)^{-2\gamma}\,\dt
	=
	B\left(\gamma,\,1-2\gamma\right)
	=
	\frac{\Gamma\left(\gamma\right)\,\Gamma\left(1-2\gamma\right)}{\Gamma\left(1-\gamma\right)}.
\end{equation*}
Recalling that
\begin{equation*}
	\mathbb E|c_n|^{2N}
	=
	\frac{1}{\Gamma\left(\gamma\right)^{2N}}\left(N!\right)^{2}\,S(n),
\end{equation*}
we obtain the asymptotic
\begin{equation*}
	\mathbb E|c_n|^{2N}
	\;\sim\;
	\frac{N!\,C_{\gamma}(N)}{n^{N(1-2\gamma)}}
	\quad\text{with}\quad
	C_{\gamma}(N)
	=
	\frac{1}{\Gamma\left(\gamma\right)^{2N}}
	\left(\frac{\Gamma\left(\gamma\right)\,\Gamma\left(1-2\gamma\right)}{\Gamma\left(1-\gamma\right)}\right)^{N}.
\end{equation*}
Using Euler's reflection formula $\Gamma\left(z\right)\Gamma\left(1-z\right)=\dfrac{\pi}{\sin\left(\pi z\right)}$,
this constant can be rewritten as
\begin{equation*}
	C_{\gamma}(N)
	=
	\left(\frac{1}{\pi}\,\Gamma\left(1-2\gamma\right)\,\sin\left(\pi\gamma\right)\right)^{N}
	=
	\kappa(\beta)^{N}.
\end{equation*}
\begin{rem}
It is instructive to verify the case $N=1$ directly:
\begin{align*}
	\mathbb E |c_n|^{2}
	&= \frac{1}{(2\pi)^2}\int_{[0,2\pi]^{2}}
	\frac{\e^{\ii n\left(\theta-\theta'\right)}}
	{|\e^{\ii\theta}-\e^{\ii\theta'}|^{2\gamma}}\,
	\dth\,\dthp
	=  \frac{1}{2\pi}\int_{-\pi}^{\pi}\frac{\e^{\ii n u}}{\left|\e^{\ii u}-1\right|^{2\gamma}}\,\du\\
	&\sim\frac{1}{2\pi\, n^{1-2\gamma}}\int_{-\infty}^{\infty}\frac{\e^{\ii t}}{|t|^{2\gamma}}\,\dt
	=\frac{\G{1-2\gamma}\,\sin{\pi\gamma}}{\pi\,n^{1-2\gamma}},
\end{align*}
which matches the general expression above.
\end{rem}
Therefore, upon defining $Z_n\de n^{\frac12-\gamma}\,c_n$, we have established the convergence of even moments
\begin{equation}\label{convmoments}
	\lim_{n\to\infty}\mathbb E\left|Z_n\right|^{2N}\;=\;\mathbb E |\mathcal N_{\mathbb C}(0,\kappa(\beta))|^{2N},
\end{equation}
and, by isotropy of $c_n$, this proves the first part of Theorem \ref{thm:CLT}.

We now deduce Theorem \ref{thm:1}. Fix $\varepsilon>0$ and choose
$N$ large enough so that $2N\varepsilon>1$. By \eqref{convmoments}, there
exists a constant $C_N>0$ such that
\begin{equation*}
\mathbb E |c_n|^{2N}
\leq C_N \,n^{-N(1-\beta^2)}
\end{equation*}
for all $n\geq 1$. Hence
\begin{equation*}
\mathbb E\left[\sum_{n\geq 1}
\left( n^{\frac{1-\beta^2}{2}-\varepsilon}|c_n| \right)^{2N}
\right] \leq C_N \sum_{n\geq 1} n^{-2N\varepsilon}
<\infty.
\end{equation*}
Therefore the series inside the expectation is finite almost surely. In
particular, its general term converges to zero almost surely, and thus
\begin{equation*}
c_n
=
o\left(
\frac{1}{n^{\frac{1-\beta^2}{2}-\varepsilon}}
\right),
\qquad n\to\infty.
\end{equation*}
The same argument applies to the negative Fourier modes.

\subsection{Proof of the critical Sobolev regularity}
We now prove that almost surely $\M\notin H^{-\beta^2/2}(\mathbb T)$. One has $\norm{\M}^2_{H^{-\beta^2/2}(\T)}=\sum_{n\in\Z} (1+n^2)^{-\beta^2/2} |c_n|^2$. Therefore, if we set
\begin{equation*}
	Y_n \de n^{1-\beta^2} |c_n|^2, \qquad S_N \de \sum_{n=1}^{N} \frac{Y_n}{n}, \qquad S \de \sum_{n=1}^{\infty} \frac{Y_n}{n} \in [0, +\infty],
\end{equation*}
it suffices to show that $S = \infty$ almost surely.

By the convergence of moments obtained in the proof of Theorem \ref{thm:CLT}, both sequences $(\E{Y_n})_n$ and $(\E{Y_n^2})_n$ converge to finite positive limits, namely $\kappa(\beta)$ and $2\kappa(\beta)^2$. Combined with the Cauchy-Schwarz inequality $\E{Y_m Y_n} \leq \sqrt{\E{Y_m^2}\E{Y_n^2}}$, this gives, as $N \to +\infty$,
\begin{equation*}
	\E{S_N} \sim \kappa(\beta)\log N
	\qquad \text{and} \qquad
	\E{S_N^2} = \sum_{m,n=1}^{N} \frac{\E{Y_m Y_n}}{mn} = O\left( (\log N)^2 \right),
\end{equation*}
so that $\liminf_{N} \E{S_N}^2 / \E{S_N^2} > 0$. Paley-Zygmund's inequality applied to $S_N$ yields, for all $N$ large enough,
\begin{equation*}
	\P\left(S_N > \frac{1}{2} \E{S_N} \right) \geq \frac{1}{4} \, \frac{\E{S_N}^2}{\E{S_N^2}} \geq \alpha
\end{equation*}
for some constant $\alpha > 0$. Since $S_N$ is nondecreasing in $N$ and $\E{S_N} \to +\infty$, for every fixed $M > 0$ and all $N$ large enough,
\begin{equation*}
	\P(S > M) \geq \P(S_N > M) \geq \P\left( S_N > \frac{1}{2}\E{S_N} \right) \geq \alpha.
\end{equation*}
Letting $M \to +\infty$ yields 
\begin{equation}\label{eq:positiveproba}
	\P(S = \infty) \geq \alpha > 0.
\end{equation}

If $f \in C^\infty(\mathbb T)$, multiplication by $e^{\ii \beta f}$ is a bounded isomorphism of $H^{s}(\mathbb T)$, for every $s\in\R$. Therefore, taking $f = -X_K$ for $K\in\N$, we see that
\begin{equation*}
	\M \in H^{-\beta^2/2}(\mathbb T) \iff e^{-\ii\beta X_K}\,\M \in H^{-\beta^2/2}(\mathbb T).
\end{equation*}
But it follows from \cite[Example 2.9]{JunnilaSaksmanWebb20} that in $H^{-s}(\T)$, for $s>1/2$, $\M = \lim_{N\to\infty} e^{\ii\beta X_N + \frac{\beta^2}{2}\mathbb E[X_N^2]}$, in probability (in fact almost surely here due to the martingale structure). By continuity of multiplication by $\e^{-\ii\beta X_K}$ on $H^{-s}(\mathbb T)$, one sees that $e^{-\ii\beta X_K}\,\M = \lim_{N\to\infty} e^{\ii\beta (X_N-X_K) + \frac{\beta^2}{2}\mathbb E[X_N^2]}$. Noting that
\[
X_N-X_K = \sum_{n=K+1}^N \frac{A_n \cos(n\cdot) + B_n \sin(n\cdot)}{\sqrt{n}}
\]
as soon as $N\geq K$, one sees that the event
\begin{equation*}
	\mathcal{A} \de \{\M \in H^{-\beta^2/2}(\mathbb T)\}
\end{equation*}
does not depend on $(A_n, B_n)_{n \leq K}$ for every $K \geq 1$. Hence $\mathcal A$ belongs to the tail $\sigma$-algebra of the i.i.d. sequence $(A_n, B_n)_{n \geq 1}$. By Kolmogorov's 0-1 law, $\P(\mathcal A) \in \{0, 1\}$. Equation \eqref{eq:positiveproba} rules out $\P(\mathcal A) = 1$, so $\P(\mathcal A) = 0$, which concludes the proof.

\section{Robustness under smooth perturbations}\label{sec:robustness}

This section is devoted to the proof of Theorem \ref{thm:robustness}. The strategy is to reduce the analysis of the Fourier modes $\widehat{\M^g}(n)$ to those of $\M$ (controlled by Theorem \ref{thm:1}) through a factorization
\begin{equation*}
	\M^g = \Phi\cdot\M,
\end{equation*}
where $\Phi$ is a (random) function regular enough so that the above product makes sense and the decay of the coefficients of $\M$ is unaffected. The proof splits into three steps: first the factorization above (Lemma \ref{lem:factor}), a convolution transfer (Lemma \ref{lem:transfer}), and the verification that the perturbation produces a sufficiently regular $\Phi$ in each of the two regimes covered by the theorem.

\subsection{Sobolev preliminaries}\label{sub:sobprelim}

We collect here some facts about Sobolev spaces on $\T$ and $\T^2$ that will be used in the sequel. We refer to \cite[Chapter 3]{SchmeisserTriebel87} for proofs and further background. For $s\in\R$ and $d\in\{1,2\}$, $H^s(\T^d)$ is defined as the space of distributions $u$ such that
\begin{equation*}
	\norm{u}_{H^s(\T^d)}^2 \de \sum_{k\in\Z^d} (1+|k|^2)^s\,|\hat u(k)|^2<\infty,
\end{equation*}
with the convention $\widehat{u}(k)\de\frac{1}{(2\pi)^d}\int_{\T^d}e^{-\ii k\cdot \theta}u(\theta)\dth$.
	\subsubsection*{Properties for $s>1/2$.} One has a continuous embedding $H^s(\T)\hookrightarrow C(\T)$ and $H^s(\T)$ is a Banach algebra: there exists $C_s>0$ such that
	\begin{equation}\label{eq:algebra}
		\norm{uv}_{H^s(\T)} \leq C_s \norm{u}_{H^s(\T)}\,\norm{v}_{H^s(\T)}.
	\end{equation}
	Moreover, for any $C^\infty$ function $F\colon\R\to\C$, the map $u\mapsto F(u)$ sends $H^s(\T)$ into itself. In particular,
	\begin{equation}\label{eq:expSob}
		u\in H^s(\T) \Longrightarrow \e^{\ii u}\in H^s(\T).
	\end{equation}
	 Multiplication by an element of $H^s(\T)$ extends continuously to $H^{-s}(\T)$ by duality: for $\Phi\in H^s(\T)$ and $M\in H^{-s}(\T)$, $\Phi\,M\in H^{-s}(\T)$ is defined by
	\begin{equation*}
		\langle \Phi\cdot M,\varphi\rangle \de \langle M,\Phi\varphi\rangle, \qquad \varphi\in H^s(\T),
	\end{equation*}
	and
	\begin{equation}\label{eq:multcont}
		\norm{\Phi \cdot M}_{H^{-s}(\T)} \leq C_s \norm{\Phi}_{H^s(\T)}\,\norm{M}_{H^{-s}(\T)}.
	\end{equation}
	The Fourier coefficients of $\Phi \cdot M$ are given by the absolutely convergent convolution
	\begin{equation}\label{eq:fourierconv}
		\widehat{\Phi \cdot M}(n) = \sum_{k\in\Z}\widehat\Phi(k)\,\widehat M(n-k), \qquad n\in\Z.
	\end{equation}
	
	\subsubsection*{Restriction property.} 
	For $s>1$ and $u\in H^s(\T^2)\subset C(\T^2)$, the restriction $u_\Delta : \theta\mapsto u(\theta,\theta)$ satisfies
	\begin{equation}\label{eq:trace}
		\norm{u_\Delta}_{H^{s-1/2}(\T)} \leq C_s \norm{u}_{H^{s}(\T)}.
	\end{equation}

\subsection{Factorization}\label{sub:factor}

Let $g$ be a continuous function and $X^g$ be the log-correlated field with covariance
\begin{equation}\label{eq:covg}
	\E{X^g(\theta)X^g(\theta')}=\log\frac{1}{|\e^{\ii\theta}-\e^{\ii\theta'}|}+ g(\theta,\theta').
\end{equation}
Assume that one has $X^g=X+Y$, where  $X$ is the exactly log-correlated field and $Y$ is a continuous centered Gaussian process on $\T$ defined on the same probability space as $X$ (not necessarily independent of it).  Let $\varphi_\varepsilon = \frac{1}{\varepsilon}\varphi(\frac{\cdot}{\varepsilon})$ where $\varphi$ is a non-negative even smooth function with $\int\varphi=1$ and support in $[-1,1]$, and write $X^g_\varepsilon\de\varphi_\varepsilon\ast X^g$. We define
\begin{equation*}
	\M^{g,\varepsilon}(\theta) \de \exp\Big(\ii\beta X^g_\varepsilon(\theta)+\frac{\beta^2}{2}\E{X^g_\varepsilon(\theta)^2}\Big),
\end{equation*}
and similarly $\M^\varepsilon$ when $g= 0$.

\begin{lem}[Factorization]\label{lem:factor}
	Suppose that there exists $\delta>0$ such that
	\begin{enumerate}[label=(\roman*)]
		\item $g\in H^{1+\delta}(\T^2)$,
		\item $Y\in H^{1/2+\delta}(\T)$ almost surely.
	\end{enumerate}
	Then, almost surely,
	\begin{equation}\label{eq:factorization-general}
		\M^g=\e^{\frac{\beta^2}{2}\,g_\Delta} \e^{\ii\beta Y}\cdot \M,
	\end{equation}
	where the right-hand side is defined via the duality pairing of \eqref{eq:multcont} in $H^{1/2+\delta}(\T)\times H^{-1/2-\delta}(\T)$ and $g_\Delta(\theta) = g(\theta,\theta)$.
\end{lem}

\begin{proof}
	The covariance of the mollified field is given by
	\begin{equation}\label{eq:bilin}
		\E{X^g_\varepsilon(\theta)X^g_\varepsilon(\theta')} = \E{X_\varepsilon(\theta)X_\varepsilon(\theta')}+g_\varepsilon(\theta,\theta'),
	\end{equation}
	with $g_\varepsilon\de(\varphi_\varepsilon\otimes\varphi_\varepsilon)\ast g$. Specializing to $\theta=\theta'$ and rearranging the renormalization yields
	\begin{equation}\label{eq:Mgsplit}
		\M^{g,\varepsilon}(\theta) = \e^{\frac{\beta^2}{2}\,g_\varepsilon(\theta,\theta)} \e^{\ii\beta Y_\varepsilon(\theta)} \,\M^\varepsilon(\theta).
	\end{equation}
	By assumption (i) and \cite[Theorem 1.1]{JunnilaSaksmanWebb20}, for any sequence $\varepsilon_n\downarrow 0$, $\M^{g,\varepsilon_n}\to\M^g$ in probability in $H^{-s}(\T)$ for every $s>1/2$. The same statement applied to $g= 0$ gives $\M^{\varepsilon_n}\to\M$.
	
	It remains to identify the limit of the right-hand side of \eqref{eq:Mgsplit}. Since $g_\varepsilon\to g$ in $H^{1+\delta}(\T^2)$, the trace inequality \eqref{eq:trace} yields $g_\varepsilon(\theta,\theta)\to g(\theta,\theta)$ in $H^{1/2+\delta}(\T)$. Together with $Y_\varepsilon\to Y$ in $H^{1/2+\delta}(\T)$ a.s. and the algebra/stability properties \eqref{eq:algebra}-\eqref{eq:expSob}, the factor
	\begin{equation*}
		\e^{\frac{\beta^2}{2}\,g_\varepsilon(\theta,\theta)} \e^{\ii\beta Y_\varepsilon(\theta)} \rightarrow
		e^{\frac{\beta^2}{2}\,g_\Delta} \e^{\ii\beta Y}\qquad \text{ a.s. in } H^{1/2+\delta}(\T), \text{ as } \varepsilon\to0.
	\end{equation*}
	The continuous action \eqref{eq:multcont} then transfers this convergence to the product against $\M^\varepsilon$ in probability in $H^{-1/2-\delta}(\T)$. The two limits in \eqref{eq:Mgsplit} must agree, which gives \eqref{eq:factorization-general}.
\end{proof}

\subsection{Transfer by convolution}\label{sub:transfer}

The next lemma is purely a statement of Fourier analysis on $\T$: a polynomial decay of the Fourier coefficients of $\M$ is preserved, after multiplication by a function in $H^{1/2+\delta}(\T)$. It will be applied to $\Phi=\e^{\beta^2 g_\Delta/2}\,\e^{\ii\beta Y}$ from Lemma \ref{lem:factor}.

\begin{lem}\label{lem:transfer}
	Let $\delta>0$, $\alpha\in(0,1/2)$. Let $\Phi\in H^{1/2+\delta}(\T)$ and $(c_n)_{n\in\Z}$ satisfy $|c_n| = O (|n|^{-\alpha})$ as $|n|\to\infty$. Then the sequence
	\begin{equation*}
		\widetilde c_n \de \sum_{k\in\Z}\widehat\Phi(k)\,c_{n-k}, \qquad n\in\Z,
	\end{equation*}
	is well-defined and
	\begin{equation*}
		\widetilde c_n= O(|n|^{-\alpha}), \qquad \text{as } |n|\to\infty.
	\end{equation*}
\end{lem}

\begin{proof}
	By Cauchy–Schwarz,
	\begin{equation}\label{eq:phi-l1}
		\sum_{k\in\Z}|\widehat\Phi(k)| \leq \Big(\sum_{k\in\Z}(1+k^2)^{-1/2-\delta}\Big)^{1/2}\norm{\Phi}_{H^{1/2+\delta}(\T)} < \infty,
	\end{equation}
	and 
	\begin{equation}\label{eq:phi-S}
		\sum_{k\in\Z}|k|\,|\widehat\Phi(k)|^2 \leq\norm{\Phi}_{H^{1/2+\delta}(\T)}^2 <\infty.
	\end{equation}
	Fix $n\geq 1$ (the case $n\leq -1$ is symmetric) and split the series defining $\widetilde c_n$ at $|k|=n/2$. By \eqref{eq:phi-l1} and the assumption on $c_n$,
	\begin{equation*}
		\bigg|\sum_{|k|\leq n/2}\widehat\Phi(k)\,c_{n-k}\bigg| \leq 	\bigg(\sum_{k\in\Z}|\widehat\Phi(k)|\bigg) \, C\Big(\frac{n}{2}\Big)^{-\alpha} = O\left(n^{-\alpha}\right).
	\end{equation*}
	By Cauchy–Schwarz with weight $\sqrt{|k|}$ and \eqref{eq:phi-S},
	\begin{equation*}
		\bigg|\sum_{|k|>n/2}\widehat\Phi(k)\,c_{n-k}\bigg|^2 \,\leq\, \norm{\Phi}_{H^{1/2+\delta}(\T)}^2\cdot \sum_{|k|>n/2}\frac{|c_{n-k}|^2}{|k|}.
	\end{equation*}
	The remaining sum is split at $|k|=2n$. For $|k|>2n$, $|n-k|\geq |k|/2$, hence
	\begin{equation*}
		\sum_{|k|>2n}\frac{|c_{n-k}|^2}{|k|} \,\leq\, C^2\sum_{|k|>2n}|k|^{-2\alpha-1} = O\left(n^{-2\alpha}\right),
	\end{equation*}
	and for $n/2<|k|\leq 2n$,
	\begin{equation*}
		\sum_{n/2<|k|\leq 2n}\frac{|c_{n-k}|^2}{|k|} \,\leq\, \frac{2C^2}{n}\sum_{|m|\leq 3n}(1+|m|)^{-2\alpha} = O\left(n^{-2\alpha}\right),
	\end{equation*}
	since $\alpha<1/2$. Combining the two contributions gives $|\widetilde c_n|=O(n^{-\alpha})$, which concludes the proof.
\end{proof}

\subsection{Proof of Theorem \ref{thm:robustness}}\label{sub:proof-robust}

The upper bound $\dim_\F\M^g\leq 1-\beta^2$ follows from the regularity result of \cite{JunnilaSaksmanWebb20}, exactly as in the discussion after Theorem \ref{thm:1}. We focus on the lower bound. The structure of the argument is the same in both regimes: produce a coupling $X^g=X+Y$ with $Y\in H^{1/2+\delta}(\T)$ a.s. for some $\delta>0$, apply Lemmas \ref{lem:factor} and \ref{lem:transfer}.

We use the Fourier-series representation of $X$
\begin{equation}\label{eq:Xfourier}
	X = \lim_{N\to\infty} X_N,\qquad X_N\de \sum_{n=1}^N \frac{A_n\cos(n\cdot)+B_n\sin(n\cdot)}{\sqrt n},
\end{equation}
where $(A_n,B_n)_{n\geq 1}$ is an i.i.d. sequence of standard real Gaussians, the convergence taking place in $H^{-s}(\T)$ a.s. for any $s>0$.

\subsubsection*{Stationary case.}

We assume here that the covariance of the field is given by $-\log|e^{\ii\theta}-e^{\ii\theta'}|+g(\theta-\theta')$ where $g\in H^{1+\delta}(\T)$. By positive semidefiniteness, one has $\hat g(0)\geq 0$ and $\frac{1}{2|n|}+\hat g(n)\geq 0$ for $n\neq 0$. Let us define
\begin{equation*}
	Y(\theta) \de \sqrt{\hat g(0)} \xi + \sum_{n\geq 1} r_n\left(A_n\cos(n\theta)+B_n\sin(n\theta)\right),
\end{equation*}
where $\xi\sim\mathcal N(0,1)$ is independent of $(A_n,B_n)_{n\geq 1}$ and
\begin{equation*}
	r_n \de \sqrt{\frac{1}{n}+2\hat g(n)} - \sqrt{\frac{1}{n}} = \frac{2\hat g(n)}{\sqrt{\frac{1}{n}+2\hat g(n)}+\sqrt{\frac{1}{n}}}.
\end{equation*}
A direct computation shows that $X+Y$ has covariance $-\log|e^{\ii\theta}-e^{\ii\theta'}|+g(\theta-\theta')$, and is thus a realization of $X^g$. Since $r_n=O(\sqrt n\,|\hat g(n)|)$, there exists $C>0$ such that
\begin{equation*}
	\E{\norm{Y}_{H^{1/2+\delta}(\T)}^2} \leq C\bigg( \hat g(0) +\sum_{n\geq 1}(1+n^2)^{1+\delta}\,|\hat g(n)|^2\bigg) <\infty,
\end{equation*}
so $Y\in H^{1/2+\delta}(\T)$ a.s. and a Fourier computation shows that $(\theta,\theta')\mapsto g(\theta-\theta')$ is in $H^{1+\delta}(\T^2)$. Assumptions (i) and (ii) of Lemma \ref{lem:factor} are therefore satisfied, with the stated $\delta$. Note that in this case $g_\Delta(\theta)=g(0)$.

\subsubsection*{Non-stationary case.}
Following \cite[Section 3]{JSW20}, we use a spectral splitting. Let $T$ be the Hilbert-Schmidt operator on $L^2(\T)$ with kernel $g\in H^{3/2+\delta}(\T^2)$, and write $T = T^+ - T^-$  where $T^\pm \de \tfrac{1}{2}(|T| \pm T)$ are positive operators with kernels $g_\pm$. We claim that $g_\pm \in H^{3/2+\delta}(\T^2)$. The proof given here  is a transposition of \cite[Lemma 3.1]{JSW20} to $\T$.

Let $K \in L^2(\T^2)$ be real-valued and symmetric, with associated self-adjoint Hilbert-Schmidt operator $T$ on $L^2(\T)$, and denote by $K_T = K$ and $K_{|T|}$ the kernels of $T$ and $|T|$ respectively. We show that for every $s \geq 0$,
\begin{equation}\label{eq:abs-sobolev}
	\|K_{|T|}\|_{H^s(\T^2)} \asymp \|K_T\|_{H^s(\T^2)}.
\end{equation}
Writing $e_n(x) \de \e^{\ii n x}$, one has $\widehat{K_T}(m,n) = \tfrac{1}{2\pi}\widehat{T e_{-n}}(m)$, and similarly for  $K_{|T|}$. Parseval's identity then yields, for every $n \in \Z$,
\[
\sum_{m \in \Z} \bigl|\widehat{K_T}(m,n)\bigr|^2 
= \frac{1}{(2\pi)^3} \bigl\langle T^2\, e_{-n}, e_{-n}\bigr\rangle
= \frac{1}{(2\pi)^3} \bigl\langle |T|^2 e_{-n}, e_{-n}\bigr\rangle
= \sum_{m \in \Z} \bigl|\widehat{K_{|T|}}(m,n)\bigr|^2,
\]
where the middle equality uses $|T|^2 = T^2$. The same holds with $m,n$ swapped by symmetry. The equivalence \eqref{eq:abs-sobolev} now follows from $(1+m^2+n^2)^s \asymp 1 + |m|^{2s} + |n|^{2s}$. Applying this with $K_T = g$ gives $g_\pm = \tfrac{1}{2}(K_{|T|} \pm g) \in H^{3/2+\delta}(\T^2)$.

By Sobolev embedding, $g_\pm$ are Hölder-continuous positive definite kernels on $\T^2$, hence covariances of centered Gaussian fields $G^\pm$ on $\T$ with Hölder-continuous realisations. Following the end of the  proof of \cite[Theorem A]{JSW20} adapted to $\T$, we may construct on a common probability space a standard log-correlated field $X$ on $\T$ and copies of $G^\pm$ such that
\begin{equation*}
X^g = X + Y, \qquad Y \de G^+ - G^-.
\end{equation*}
We claim that $G^\pm\in H^{1/2+\delta_0}(\T)$ a.s. for every $\delta_0<\delta/2$. Indeed, one has
\begin{equation*}
	\E{|\widehat{G^{\pm}}(n)|^2}=\widehat{g_{\pm}}(n,-n),\qquad n\in\Z,
\end{equation*}
so that, by Cauchy–Schwarz, for any $s\in\R$ and $\tau>2s+1/2$,
\begin{align*}
	\E{\norm{G^{\pm}}_{H^s(\T)}^2}
	&=\sum_{n\in\Z}(1+n^2)^s\,\widehat{ g_{\pm}}(n,-n)\\
	&\leq C\bigg(\sum_{n\in\Z}(1+n^2)^{2s-\tau}\bigg)^{\frac12}\bigg(\sum_{n\in\Z} (1+n^2)^\tau \,\widehat{g_{\pm}}(n,-n)^2\bigg)^{\frac12}\\
	&\leq C\bigg(\sum_{n\in\Z}(1+n^2)^{2s-\tau}\bigg)^{\frac12}\,\norm{g_{\pm}}_{H^\tau(\T^2)}.
\end{align*}
This gives $\E{\norm{G^\pm}_{H^{1/2+\delta_0}(\T)}^2}<\infty$ for every $\delta_0<\delta/2$, hence $Y\in H^{1/2+\delta_0}(\T)$ a.s.

Therefore, assumptions (i) and (ii) of Lemma \ref{lem:factor} are satisfied, with $\delta$ replaced by $\delta_0$.

\subsubsection*{Conclusion.}
In each case, Lemma \ref{lem:factor} provides $\delta_0>0$ such that
\begin{equation}\label{eq:factor-final}
	\M^g =\Phi\cdot\M\quad\text{a.s.},\qquad
	\Phi \de\e^{\frac{\beta^2}{2} g_\Delta}\,\e^{\ii\beta Y}\in H^{1/2+\delta_0}(\T) \text{ a.s.}
\end{equation}
By \eqref{eq:fourierconv} and \eqref{eq:factor-final},
\begin{equation*}
	\widehat{\M^g}(n)=\sum_{k\in\Z}\widehat\Phi(k)\,c_{n-k},\qquad n\in\Z,
\end{equation*}
where $c_n=\widehat\M(n)$. By Theorem \ref{thm:1}, $|c_n| = O\big(|n|^{-(1-\beta^2)/2+\eta}\big)$ a.s. for every $\eta>0$. Lemma \ref{lem:transfer} then gives, for every $\eta>0$,
\begin{equation*}
	\Big|\widehat{\M^g}(n)\Big|= O\left(|n|^{-(1-\beta^2)/2+\eta}\right)\qquad\text{a.s.},
\end{equation*}
which concludes the proof of the lower bound. \qed

\section{Convergence of the process}\label{sec:mom2}
We now turn to the study of mixed moments of the Fourier coefficients
\begin{equation*}
	M_n(\ell,m)\de \E{\prod_{j=0}^k c_{n+j}^{\ell_j}\,\overline{c_{n+j}}^{m_j}}.
\end{equation*}
We first treat the modulus case and then address the general mixed moment in Section \ref{sub:mix}.
This analysis requires additional properties of Jack polynomials, which we now collect.

\subsection{Preliminary results}
\subsection*{Pieri formula.}
The Pieri formula gives the explicit decomposition of the product $e_p P_\mu$ on the Jack polynomial basis, where $e_p$ is  the elementary symmetric polynomial of degree $p$. It takes the following form
\begin{equation}\label{eq:Pieriformula}
	e_{p}\,P_{\mu}=\sum_{\tau\;:\; \tau/\mu\; p\text{-vertical strip}}\psi'_{\tau/\mu}\,P_{\tau}.
\end{equation}
We refer to Appendix \ref{partition} for the definitions of skew diagrams and vertical strips.
If $\tau$ and $\mu$ are partitions such that $\tau\supset\mu$, let $C_{\tau/\mu}$ (resp. $R_{\tau/\mu}$) denote the union of the columns (resp. rows) that intersect $\tau-\mu$. Thus $C_{\tau/\mu}$ consists of all boxes of $\tau$ lying in a column in which a box has been added, while $R_{\tau/\mu}$ consists of all boxes of
$\tau$ lying in a row in which a box has been added. The coefficient $\psi'_{\tau/\mu}$, which we call the \emph{Pieri coefficient}, is given by
\begin{equation}\label{eq:psi'}
	\psi'_{\tau/\mu}\,= \prod_{s\in C_{\tau/\mu}- R_{\tau/\mu}} \frac{b_\tau(s)}{b_\mu(s)},
\end{equation}
where
\begin{equation}\label{eq:blambda}
	b_\lambda(s) = \frac{c_\lambda(s)}{c'_\lambda(s)} = \frac{a_\lambda(s)+\gamma\,l_\lambda(s)+\gamma}{a_\lambda(s)+\gamma\,l_\lambda(s)+1}.
\end{equation}
There are in fact four Pieri formulas but only this one will be needed below. For more details on those formulas, one can refer to \cite[p340]{Macdonald95}.

\begin{defi}
	If $\lambda$ is a partition, let $\Delta\lambda_i=\lambda_i-\lambda_{i+1}$, with the convention $\lambda_{N+1}=0$. The gap of the partition is defined as
	\begin{equation*}
		\mathrm{gap}(\lambda) = \min_{i\leq N} \Delta\lambda_i.
	\end{equation*}
\end{defi}

\begin{defi}
	For $\ell_1,\cdots,\ell_k\in \mathbb N$, an \emph{$(\ell_1,\dots,\ell_k)$-shape} of height $N$ is a vector
	$\sigma=(\sigma_1,\dots,\sigma_N)\in\{0,1,\dots,k\}^N$ such that
	\begin{equation*}
		|\{i:\;\sigma_i=r\}|=\ell_r,\qquad r=1,\dots,k.
	\end{equation*}
	Given such $\sigma$ and a partition $\lambda$
	satisfying $\mathrm{gap}(\lambda)\geq k$, we
	define a new partition $\nu=\lambda+ \sigma$ by
	\begin{equation*}
		\nu_i \de \lambda_i + \sigma_i,\qquad i=1,\dots,N.
	\end{equation*}
	In this case, we say that the skew diagram $\nu/\lambda$
	has shape $\sigma$. We also say that $\nu$ is obtained from $\lambda$ by adding an $(\ell_1,\dots,\ell_k)$-shape if
	$\nu=\lambda+\sigma$ for some $(\ell_1,\dots,\ell_k)$-shape $\sigma$. Let us denote $|\sigma| = \sum_{i=1}^N \sigma_i$ and
	\begin{equation}
		\mathcal A_\ell = \left\{\sigma\in\{0,1,\dots,k\}^N\; : \; \sigma \text{ is an }\ell\text{-shape}\right\}.
	\end{equation}
\end{defi}

\begin{figure}[H]
	\centering
	\subfigure{
		\begin{tikzpicture}[scale=0.45]
			\begin{scope}
				\foreach \row/\len in {0/15,1/13,2/10,3/8,4/5,5/3,6/1}{
					\foreach \x in {0,...,\numexpr\len-1}{
						\draw[thick] (\x,-\row) rectangle ++(1,-1);
					}
				}
				\node at (5.5,1) {\large $\lambda$};
			\end{scope}
			
			\draw[->,thick] (13.8,-3.5) -- (16.2,-3.5);
			\node at (15,-2.8) {$+\,\sigma\,$};

			\begin{scope}[shift={(17,0)}]
				
				\foreach \row/\len in {0/15,1/13,2/10,3/8,4/5,5/3,6/1}{
					\foreach \x in {0,...,\numexpr\len-1}{
						\draw[thick] (\x,-\row) rectangle ++(1,-1);
					}
				}
				
				\foreach \coord in {(15,0)}{
					\fill[blue!25] \coord rectangle ++(1,-1);
					\draw[thick] \coord rectangle ++(1,-1);
				}
				\foreach \coord in {(13,-1)}{
					\fill[blue!70] \coord rectangle ++(1,-1);
					\draw[thick] \coord rectangle ++(1,-1);
				}
				\foreach \coord in {(14,-1)}{
					\fill[blue!70] \coord rectangle ++(1,-1);
					\draw[thick] \coord rectangle ++(1,-1);
				}

				\foreach \coord in {(8,-3)}{
					\fill[blue!25] \coord rectangle ++(1,-1);
					\draw[thick] \coord rectangle ++(1,-1);
				}

				\foreach \coord in {(3,-5)}{
					\fill[blue!25] \coord rectangle ++(1,-1);
					\draw[thick] \coord rectangle ++(1,-1);
				}
				
				\foreach \coord in {(10,-2),(11,-2)}{
					\fill[blue!70] \coord rectangle ++(1,-1);
					\draw[thick] \coord rectangle ++(1,-1);
				}

				\node at (6.5,1) {\large $\nu$};

			\end{scope}
			
		\end{tikzpicture}
	}
	
	\caption{An example of partition $\nu=\lambda + \sigma$ obtained by adding $\sigma$ of shape $\ell_1=3$ and $\ell_2=2$.}
	\label{fig:+s}
\end{figure}
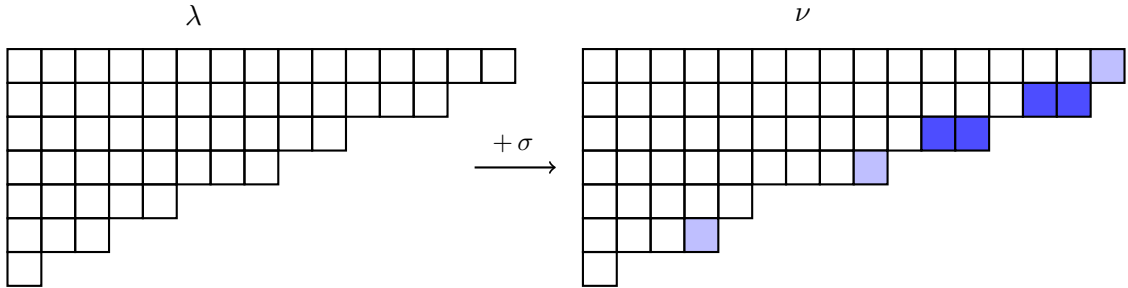
The two following lemmas give the estimates of the Pieri coefficients in the regime where the spacings between the parts of $\lambda$ become large
\begin{equation}
	\mathrm{gap}(\lambda)\longrightarrow +\infty.
\end{equation}
 We call this regime the \emph{large gap regime}.
\begin{lem}\label{lem:pierito1} If $\sigma$ is a vertical strip, the Pieri coefficients satisfy the following properties:
	\begin{equation*}
		\psi'_{\lambda+n+\sigma/\lambda+n}=\psi'_{\lambda+\sigma/\lambda},
	\end{equation*}
	and
	\begin{equation*}
		\psi'_{\lambda+\sigma/\lambda}=1+O\left(\frac{1}{\mathrm{gap}(\lambda)}\right).
	\end{equation*}
\end{lem}
\begin{proof}
	The first equality is a consequence of the shift property \eqref{eq:shift}. For the second one, we use the formulas \eqref{eq:psi'} and \eqref{eq:blambda} to get
	\begin{equation*}
		\psi'_{\lambda+\sigma/\lambda}\,
		=
		\prod_{s\in C_{\lambda+\sigma/\lambda}- R_{\lambda+\sigma/\lambda}}  \frac{c_{\lambda+\sigma}(s)}{c'_{\lambda+\sigma}(s)}\, \frac{c'_\lambda(s)}{c_\lambda(s)}.
	\end{equation*}
	First, note that the product has at most $(N-1)|\sigma|$ terms. So at fixed shape $\sigma$, it is sufficient to prove that each term is $1+O\left(\frac{1}{\mathrm{gap}(\lambda)}\right)$. And
	\begin{align*}
		\frac{c_{\lambda+\sigma}(s)}{c'_{\lambda+\sigma}(s)}\, \frac{c'_\lambda(s)}{c_\lambda(s)}
		&=\frac{a_{\lambda+\sigma}(s)+\gamma l_{\lambda+\sigma}(s)+\gamma}{a_{\lambda+\sigma}(s)+\gamma l_{\lambda+\sigma}(s)+1}
		\; \frac{a_{\lambda}(s)+\gamma l_{\lambda}(s)+1}{a_{\lambda}(s)+\gamma l_{\lambda}(s)+\gamma}\\
		&=\frac{a_{\lambda}(s)+\gamma( l_{\lambda}(s)+1)+\gamma}{a_{\lambda}(s)+\gamma (l_{\lambda}(s)+1)+1}
		\; \frac{a_{\lambda}(s)+\gamma l_{\lambda}(s)+1}{a_{\lambda}(s)+\gamma l_{\lambda}(s)+\gamma}\\
		&=\left(1-\frac{1-\gamma}{a_{\lambda}(s)+\gamma (l_{\lambda}(s)+1)+1}\right)\left(1+\frac{1-\gamma}{a_{\lambda}(s)+\gamma l_{\lambda}(s)+\gamma}\right)\\
		&= 1+O\left(\frac{1}{\mathrm{gap}(\lambda)}\right).
	\end{align*}
\end{proof}
Similarly, in the large gap regime, one has the following equivalent for ratios of $c$'s.
\begin{lem}\label{lem:ratio}
	Let $\sigma$ be a fixed shape, then we have
	\begin{equation*}
		\frac{c'_{\lambda+\sigma}}{c'_{\lambda}}\,\frac{c_{\lambda}}{c_{\lambda+\sigma}}=1+O\left(\frac{1}{\mathrm{gap}(\lambda)}\right).
	\end{equation*}
\end{lem}
\begin{proof}
	By decomposing the shape $0\subset \sigma^{(1)}\subset\cdots\subset\sigma^{(|\sigma|)}=\sigma$, one sees that it is sufficient to prove the result for a shape $\sigma$ of size $1$, i.e. just adding one square to row $i\in[N]$. In this particular case, the only affected cells in the new coefficient $c_{\lambda+\sigma}$ are those in $C_{\lambda+\sigma/\lambda}\cup R_{\lambda+\sigma/\lambda}$, which are those in row $i$ and above the extra cell. Consequently,
	\begin{align*}
		\frac{c_{\lambda+\sigma}}{c_{\lambda}}
		&=  \prod_{s\in C_{\lambda+\sigma/\lambda}\cup R_{\lambda+\sigma/\lambda}}  \frac{c_{\lambda+\sigma}(s)}{c_{\lambda}(s)} \\
		&= \ga\; \prod_{s\in C_{\lambda+\sigma/\lambda} - \{\sigma\}}  \frac{c_{\lambda+\sigma}(s)}{c_{\lambda}(s)} \; \prod_{s\in R_{\lambda+\sigma/\lambda}- \{\sigma\}}  \frac{c_{\lambda+\sigma}(s)}{c_{\lambda}(s)} \\
		&= \ga \;\prod_{s\in C_{\lambda+\sigma/\lambda} - \{\sigma\}}  \frac{c_{\lambda}(s)+\gamma}{c_{\lambda}(s)} \; \prod_{s\in R_{\lambda+\sigma/\lambda}- \{\sigma\}}  \frac{c_{\lambda}(s)+1}{c_{\lambda}(s)} \\
		&= \ga \;\prod_{s\in C_{\lambda+\sigma/\lambda} - \{\sigma\}}  \left(1+O\left(\frac{1}{\mathrm{gap}(\lambda)}\right)\right) \; \prod_{s\in R_{\lambda+\sigma/\lambda}- \{\sigma\}}  \frac{c_{\lambda}(s)+1}{c_{\lambda}(s)}.
	\end{align*}
	The last term requires a bit more care due to the leg-lengths. We cut the row $i$ in parts with fixed leg-length:
	\begin{figure}[H]
		\usetikzlibrary{positioning}
		
		\[
		\begin{tikzpicture}[
			baseline=(current bounding box.center),
			box/.style={draw, rectangle, minimum height=1.4em, inner xsep=5pt, inner ysep=2pt},
			node distance=2pt
			]
			\node[] (L) {$\lambda_i +1:$};
			\node[box, right=2pt of L] (A) {$1\;\cdots \;\lambda_N$};
			\node[box, right=2pt of A] (B) {$\lambda_N + 1\;\cdots \;\lambda_{N-1}$};
			\node[box, right=2pt of B] (C) {$\cdots$};
			\node[box, right=2pt of C] (D) {$\lambda_{i+1}\; + 1\cdots \;\lambda_i$};
			\node[box, right=2pt of D] (S) {$\sigma$};
		\end{tikzpicture}
		\]
	\end{figure}
	One ends up with	
	\begin{align*}
		\prod_{s\in R_{\lambda+\sigma/\lambda}- \{\sigma\}}  \frac{c_{\lambda}(s)+1}{c_{\lambda}(s)}
		&= \prod_{k=0}^{N-i} \prod_{j=\lambda_{N-k+1}+1}^{\lambda_{N-k}}\frac{\lambda_i-j+\ga(N-i-k)+\ga+1}{\lambda_i-j+\ga(N-i-k)+\ga}\\
		&= \prod_{k=0}^{N-i} \frac{\lambda_i-(\lambda_{N-k+1}+1)+\ga(N-i-k)+\ga+1}{\lambda_i-\lambda_{N-k}+\ga(N-i-k)+\ga}.
	\end{align*}
	Therefore
	\begin{align*}
		\frac{c_{\lambda+\sigma}}{c_{\lambda}}\,\frac{c'_{\lambda}}{c'_{\lambda+\sigma}}
		&=  \gamma\left(1+O\left(\frac{1}{\mathrm{gap}(\lambda)}\right)\right) \times\\
		&\prod_{k=0}^{N-i} \frac{\lambda_i-(\lambda_{N-k+1}+1)+\ga(N-i-k)+\ga+1}{\lambda_i-\lambda_{N-k}+\ga(N-i-k)+\ga} \frac{\lambda_i-\lambda_{N-k}+\ga(N-i-k)+1}{\lambda_i-(\lambda_{N-k+1}+1)+\ga(N-i-k)+1+1}\\
		&=  \gamma\left(1+O\left(\frac{1}{\mathrm{gap}(\lambda)}\right)\right)\,\frac{1}{\ga}\, \frac{\Delta\lambda_i+\ga}{\Delta\lambda_i+1}\\
		&\prod_{k=0}^{N-i-1} \frac{\lambda_i-\lambda_{N-k+1}+\ga(N-i-k)+\ga}{\lambda_i-\lambda_{N-k+1}+\ga(N-i-k)+1}\;
		\frac{\lambda_i-\lambda_{N-k}+\ga(N-i-k)+1}{\lambda_i-\lambda_{N-k}+\ga(N-i-k)+\ga}\\
		&=1+O\left(\frac{1}{\mathrm{gap}(\lambda)}\right).
	\end{align*}

\end{proof}

The next elementary lemma allows one to replace, in the large gap regime, a bounded coefficient by its asymptotic value inside the partition sum. It will be used repeatedly in the proof of the joint convergence results. Recall that 
\[
G_\lambda(n)=\prod_{i=1}^{N} F\big((N-i+1)\gamma+\lambda_i\big)\,
F\big((N-i+1)\gamma+\lambda_i+n\big).
\]
\begin{lem}\label{lem:asymp}
	Let  $(a_\lambda(n))_{\lambda,n}$ be a family of real numbers such that
	\[
	\sup_{\lambda,n}|a_\lambda(n)|<\infty,
	\]
	and
	\[
	\lim_{M\to\infty}\;\sup_{n\ge 1}\;\sup_{\gap(\lambda)\ge M}\;
	|a_\lambda(n)-1|=0.
	\]
	Then, as  $n\to\infty$,
	\[
	\sum_\lambda a_\lambda(n)\,G_\lambda(n)\sim \sum_\lambda G_\lambda(n).
	\]
\end{lem}

\begin{proof}
	Fix  $\varepsilon>0$ and decompose
	\[
	\sum_\lambda \left(a_\lambda(n)-1\right)G_\lambda(n)
	=
	\sum_{\gap(\lambda)\le \varepsilon n}\left(a_\lambda(n)-1\right)G_\lambda(n)
	+
	\sum_{\gap(\lambda)> \varepsilon n}\left(a_\lambda(n)-1\right)G_\lambda(n).
	\]
	Hence, using  $S(n)=\sum_\lambda G_\lambda(n)$, one gets
	\[
	\left|\sum_\lambda a_\lambda(n)G_\lambda(n)-S(n)\right|
	\leq
	C\sum_{\gap(\lambda)\le \varepsilon n}G_\lambda(n)
	+
	\sup_{\gap(\lambda)>\varepsilon n}\left|a_\lambda(n)-1\right|\,S(n),
	\]
	where  $C=\sup_{\lambda,n}\left|a_\lambda(n)-1\right|<\infty$.
	
	Arguing as in the proof of Proposition \ref{pro:S}, one obtains the following negligibility estimate for the contribution of partitions with small gap:
	\[
	\lim_{\varepsilon\downarrow 0}\limsup_{n\to\infty}
	\frac{1}{S(n)}\sum_{\gap(\lambda)\le \varepsilon n}G_\lambda(n)=0.
	\]
	On the other hand, for each fixed  $\varepsilon>0$, the assumption on  $a_\lambda(n)$ gives
	\[
	\sup_{\gap(\lambda)>\varepsilon n}\left|a_\lambda(n)-1\right|\xrightarrow[n\to\infty]{}0.
	\]
	Dividing by  $S(n)$, taking  $n\to\infty$, and then  $\varepsilon\downarrow0$, we obtain
	\[
	\frac{1}{S(n)}\sum_\lambda a_\lambda(n)G_\lambda(n)\longrightarrow 1,
	\]
	which proves the claim.
\end{proof}

Finally, a straightforward computation gives the following form for the norm of $P_\lambda$:
\begin{equation}\label{lem:normP}
	\nor{P_\lambda}^2
	=
	K(N,\ga)\,C(N,\ga)\;\frac{c'_\lambda}{c_\lambda}\;
	\prod_{i=1}^N \F\left((N-i+1)\ga+\lambda_i\right),
\end{equation}
with $K(N,\ga)\de\frac{\Gamma(1+N\ga)}{\Gamma(1+\ga)^N}$ and
$C(N,\ga)\de\frac{\ga^{N-1}(N-1)!}{\Gamma(N\ga)}$.

\subsection{Joint convergence: $k=1$}
Fix $N\ge 1$ and $0\le p\le N$ and consider
\begin{equation*}
	\E{ |c_n|^{2(N-p)}|c_{n+1}|^{2p}}.
\end{equation*}
The same arguments as in Section \ref{sub:jackexp} lead to the following expression for the joint moment:
\begin{align*}
	\E{ |c_n|^{2(N-p)}|c_{n+1}|^{2p}}
	&=\frac{1}{(2\pi)^{2N}}\int_{[0,2\pi]^{2N}}
	\e^{\ii n\sum_{i\leq N}(\theta_i-\theta_i')}
	\e^{\ii \sum_{i\leq p}(\theta_i-\theta_i')}
	\prod_{ i<j} |\e^{\ii\theta_i}-\e^{\ii\theta_j}|^{2\gamma}\,|\e^{\ii\theta'_i}-\e^{\ii\theta'_j}|^{2\gamma}\\
	&\qquad\qquad\qquad\qquad\qquad\qquad\qquad\qquad\qquad\prod_{i,j}|\e^{\ii\theta_i}- \e^{\ii\theta'_j}|^{-2\gamma}\,
	\dth\,\dthp\\
	&=\sum_{\lambda,\nu}\frac{c_\lambda}{c'_\lambda}\frac{c_\nu}{c'_\nu}\,
	\frac{1}{(2\pi)^{2N}}\int_{[0,2\pi]^{2N}} \e^{\ii \sum_{i\leq p}\theta_i}P_{\lambda+n}(\e^{\ii\theta})\, 
	\e^{-\ii \sum_{i\leq p}\theta_i'}P_{\lambda+n}(\e^{-\ii\theta'})\\
	&\qquad\qquad\qquad\qquad\qquad P_\nu(\e^{-\ii\theta})\,P_\nu(\e^{\ii\theta'})\prod_{ i<j} |\e^{\ii\theta_i}-\e^{\ii\theta_j}|^{2\gamma}\,|\e^{\ii\theta'_i}-\e^{\ii\theta'_j}|^{2\gamma}\,\,\dth\,\dthp\nonumber\\
	&=\sum_{\lambda,\nu}\frac{c_\lambda}{c'_\lambda}\frac{c_\nu}{c'_\nu}\,
	\left|\left\langle x_{1}\cdots x_{p}\,P_{\lambda+n},\,P_{\nu}\right\rangle_{\gamma}\right|^2,
\end{align*}
where $x_{1},\dots,x_{p}$ are the $p$ variables carrying the $(n+1)$-phase.
In order to apply the Pieri formula \eqref{eq:Pieriformula}, we first symmetrize the expression. By the invariance of the Selberg inner product and the polynomials $P_{\mu}$ under permutations, all choices of $p$ variables among the $x_i$ contribute equally. Averaging over these $\binom{N}{p}$ choices therefore yields the following expression
\begin{equation}\label{eq:ep}
	\left\langle x_{1}\cdots x_{p}\,P_{\lambda+n},\,P_{\nu}\right\rangle_{\gamma}
	=
	\frac{1}{\binom{N}{p}}\left\langle e_{p}\,P_{\lambda+n},\,P_{\nu}\right\rangle_{\gamma}.
\end{equation}
The orthogonality of the Jack polynomials implies that
\begin{equation*}\label{eq:Pieri}
	\left\langle e_{p}\,P_{\lambda+n},\,P_{\nu}\right\rangle_{\gamma}
	=
	\begin{cases}
		\psi'_{\nu/(\lambda+n)}\,\left\|P_{\nu}\right\|_{\gamma}^{2},
		& \text{if } \nu/(\lambda+n) \text{ is a vertical strip of size } p, \\[0.3em]
		0, & \text{otherwise}.
	\end{cases}
\end{equation*}
This leads to
\begin{align*}
	\E{|c_n|^{2(N-p)} |c_{n+1}|^{2p}}
	&=\frac{1}{\binom{N}{p}^2}\,
	\sum_{\lambda}\,\frac{c_{\lambda}}{c'_{\lambda}}\,\sum_{\sigma}\,
	\frac{c_{\lambda+n+\sigma}}{c'_{\lambda+n+\sigma}} \psi^{\prime\,2}_{\lambda+n+\sigma/\lambda+n}\nor{P_{\lambda+n+\sigma}}^4,
\end{align*}
where the sum runs over all vertical strips $\sigma$ of size $p$. Therefore, with the notations in \eqref{lem:normP}
\begin{align*}
	\E{|c_n|^{2(N-p)} |c_{n+1}|^{2p}}
	&=\frac{1}{\binom{N}{p}^2}\, K(N,\ga)\,C(N,\ga)\,
	\sum_{\lambda}\frac{c_{\lambda}}{c'_{\lambda}}\,\sum_{\sigma}
	\psi^{\prime\,2}_{\lambda+\sigma/\lambda}\\
	&\qquad
	\prod_i \F((N-i+1)\ga+\lambda_i+n+\sigma_i)\,\nor{P_{\lambda+\sigma}}^2\\
	&= \frac{1}{\binom{N}{p}^2}\, K(N,\ga)^2\,C(N,\ga)^2\,\sum_\sigma
	\sum_{\lambda}\frac{c_{\lambda}}{c'_{\lambda}}\,\frac{c'_{\lambda+\sigma}}{c_{\lambda+\sigma}}\,\psi^{\prime\,2}_{\lambda+\sigma/\lambda}\\
	&\qquad
	\prod_i \frac{F((N-i+1)\gamma+\lambda_i+n+\sigma_i)F((N-i+1)\gamma+\lambda_i+\sigma_i)}{F((N-i+1)\gamma+\lambda_i+n)F((N-i+1)\gamma+\lambda_i)}\,G_\lambda(n)\\
	&=\frac{1}{\binom{N}{p}^2}\, K(N,\ga)^2\,C(N,\ga)^2
	\sum_\sigma\sum_{\lambda}a_\lambda^\sigma(n) G_\lambda(n),
\end{align*}
 with
\begin{equation*}
	a^\sigma_\lambda(n)
	=
	 \frac{c_{\lambda}}{c'_{\lambda}}\,\frac{c'_{\lambda+\sigma}}{c_{\lambda+\sigma}}\;\psi^{\prime\,2}_{\lambda+\sigma/\lambda}\;\prod_i \frac{F((N-i+1)\gamma+\lambda_i+n+\sigma_i)F((N-i+1)\gamma+\lambda_i+\sigma_i)}{F((N-i+1)\gamma+\lambda_i+n)F((N-i+1)\gamma+\lambda_i)}.
\end{equation*}
Using Lemma \ref{lem:pierito1}, Lemma \ref{lem:ratio}, and a straightforward estimate on the product term, we see that
	\[
	a^\sigma_\lambda(n)=1+O\left(\frac{1}{\mathrm{gap}(\lambda)}\right),
	\]
	uniformly in $n$.
Hence, by Lemma \ref{lem:asymp}, each of the $\binom{N}{p}$ terms indexed by $\sigma$ contributes equally at leading order, yielding
\begin{align*}
	\E{|c_n|^{2(N-p)} |c_{n+1}|^{2p}}
	&=\frac{1}{\binom{N}{p}^2}\, K(N,\ga)^2\,C(N,\ga)^2\,
	\sum_\sigma\sum_{\lambda}a_\lambda^\sigma(n) G_\lambda(n)\\
	&\sim (N-p)! \, p!\, \kappa(\beta)^N\; n^{-(1-2\ga)N},
\end{align*}
as $n\rightarrow\infty$. We recognize the mixed moments of two independent $\mathcal N_{\mathbb C}(0,\kappa(\beta))$ random variables.

\subsection{Joint convergence: $k=2$}

Before turning to the general case, let us briefly discuss the case $k=2$ at an
informal level. It already reveals the main mechanism of the argument.

Fix non-negative integers $\ell_0,\ell_1,\ell_2$ and set
\[
N=\ell_0+\ell_1+\ell_2.
\]
We consider
\[
\E{\left|c_n\right|^{2\ell_0}\left|c_{n+1}\right|^{2\ell_1}\left|c_{n+2}\right|^{2\ell_2}}.
\]
Exactly as in the case $k=1$, after applying Stanley's Cauchy identity and the
shift property of Jack polynomials, one is led to scalar products of the form
\[
\big\langle (x_{\ell_0+1}\cdots x_{\ell_0+\ell_1})
(x_{\ell_0+\ell_1+1}\cdots x_N)^2\,P_{\lambda+n},P_\nu\big\rangle_\gamma .
\]
The relevant symmetric polynomial is therefore $Q_{\ell_1,\ell_2}$, defined as the
symmetrization of the monomial
\[
(x_{\ell_0+1}\cdots x_{\ell_0+\ell_1})(x_{\ell_0+\ell_1+1}\cdots x_N)^2.
\]

To bring the Pieri formula into play, one would like to rewrite
$Q_{\ell_1,\ell_2}$ in terms of the elementary symmetric polynomials. This is
still manageable for small values of $(\ell_1,\ell_2)$, but it is not the right
point of view for the general case. Instead, we introduce the generating series
\[
E(\alpha)\de\prod_{i=1}^N (1+\alpha x_i)=\sum_{r=0}^N \alpha^r e_r(x).
\]
Indeed,
\[
E(\alpha)E(\beta)
= \prod_{i=1}^N (1+\alpha x_i)(1+\beta x_i)
= \prod_{i=1}^N (1+ux_i+v x_i^2),
\]
where $u=\alpha+\beta$ and $v=\alpha\beta$. Expanding the last product yields
\[
E(\alpha)E(\beta)
=\sum_{n_1+n_2\le N} u^{n_1}v^{n_2} Q_{n_1,n_2}(x),
\]
so that $Q_{\ell_1,\ell_2}$ is exactly the coefficient of $u^{\ell_1}v^{\ell_2}$.
Pieri formula applied to $E(\alpha)=\sum_{r=0}^N \alpha^r e_r(x)$ gives
\[
E(\alpha)P_\lambda=\sum_{\mu\supset\lambda}\alpha^{|\mu|-|\lambda|}
\psi'_{\mu/\lambda}P_\mu,
\]
where the sum runs over all vertical strips.
Applying the Pieri formula successively for $\alpha$ and $\beta$ gives a
decomposition of $Q_{\ell_1,\ell_2}P_\lambda$ into Jack polynomials obtained by
adding two vertical strips. In the large gap regime, Lemma \ref{lem:pierito1}
shows that the Pieri coefficients are asymptotically equal to $1$. This is the
key simplification, and in the next section we will use it to show that asymptotically
\[
Q_{\ell_1,\ell_2} P_\lambda = \sum_{\sigma\; (\ell_1,\ell_2)\text{-shape}} P_{\lambda+\sigma}.
\]
Returning to the moment expansion, this means that only partitions of the form
\[
\nu=\lambda+n+\sigma,
\]
with $\sigma$ of $(\ell_1,\ell_2)$-shape, contribute at leading order. The number
of such shapes is $\binom{N}{\ell_0,\ell_1,\ell_2} = \frac{N!}{\ell_0 ! \ell_1 ! \ell_2!}$,
and each of them gives the same leading asymptotics. We therefore recover
\[
\E{\left|c_n\right|^{2\ell_0}\left|c_{n+1}\right|^{2\ell_1}\left|c_{n+2}\right|^{2\ell_2}}
\sim\; \ell_0!\ell_1!\ell_2!\;\kappa(\beta)^N\, n^{-N(1-2\gamma)}.
\]
Let us now pass directly to the general case.

\subsection{General case $k$}\label{sec:generalk}

Fix an integer $k\geq 1$ and non-negative integers $\ell_0,\dots,\ell_k$ such that
\[
\ell_0+\dots+\ell_k=N.
\]
We write
\[
\ell=(\ell_0,\dots,\ell_k),\qquad
\text{wt}(\ell)\de\sum_{r=1}^k r\ell_r.
\]
We consider here the mixed modulus moment
\begin{equation*}
\mathbb E\bigg[\prod_{j=0}^k |c_{n+j}|^{2\ell_j}\bigg].
\end{equation*}
Let $Q_\ell$ be the symmetrization of a monomial with $\ell_j$ variables of degree $j$, for $0\le j\le k$. Equivalently,
\begin{equation*}
	Q_{\ell}(x)
	\de
	\sum_{(I_0,\dots,I_k)}\; \prod_{j=0}^k x_{I_j}^{j},
\end{equation*}
where the sum runs over all ordered partitions $(I_0,\dots,I_k)$ of $\{1,\dots,N\}$ such that $|I_j|=\ell_j$, and where $x_I \de \prod_{i\in I}x_i$.
By symmetry of the Selberg inner product, the same argument as in the cases $k=1,2$ gives
\begin{equation}\label{eq:kmoment}
	\mathbb E\bigg[\prod_{j=0}^k |c_{n+j}|^{2\ell_j}\bigg]
	=
	\frac{1}{\binom{N}{\ell}^2}\;
	\sum_{\lambda,\nu}\;
	\frac{c_\lambda}{c'_\lambda}\frac{c_\nu}{c'_\nu}
	\;
	\Big|\big\langle Q_\ell P_{\lambda+n},P_\nu\big\rangle_\gamma\Big|^2.
\end{equation}

\subsection*{Exact expansion in the Jack basis.}
Let
\[
S_\ell\de\Big\{\sigma\in\{0,\dots,k\}^N: \sum_{i=1}^N \sigma_i=\text{wt}(\ell)\Big\}.
\]
For every partition $\lambda$ and every $\sigma\in S_\ell$, let
$A^\ell_\lambda(\sigma)$ be the coefficient of $P_{\lambda+\sigma}$ in the decomposition of
$Q_\ell P_\lambda$ on the Jack basis, with the convention that
$A^\ell_\lambda(\sigma)=0$ whenever $\lambda+\sigma$ is not a partition. By the shift property,
these coefficients do not depend on $n$, and one has the exact expansion
\begin{equation}\label{eq:QPexp}
	Q_\ell P_{\lambda+n}
	=
	\sum_{\sigma\in S_\ell} A^\ell_\lambda(\sigma)\,P_{\lambda+n+\sigma}.
\end{equation}

In order to compute the coefficients $A^\ell_\lambda(\sigma)$, it is easier to consider the action of more general polynomials on $P_\lambda$. Introduce formal parameters $t_1,\dots,t_k$ and remark that
\[
\prod_{i=1}^N (1+t_1x_i+\cdots+t_kx_i^k)
=
\sum_{n_0+\cdots+n_k=N} t_1^{n_1}\cdots t_k^{n_k}\,Q_{n_0,\dots,n_k}(x).
\]
The following factorization will be very useful
\[
1+t_1x+\cdots+t_kx^k=\prod_{j=1}^k(1+\alpha_jx).
\]
In other words, $t_r=e_r(\alpha_1,\dots,\alpha_k)$, and
\[
\prod_{i=1}^N (1+t_1x_i+\cdots+t_kx_i^k)
=
\prod_{j=1}^k E(\alpha_j),
\qquad
E(\alpha)=\prod_{i=1}^N(1+\alpha x_i)=\sum_{r=0}^N \alpha^r e_r.
\]
Therefore,
\begin{equation}\label{eq:E-Q}
	\prod_{j=1}^k E(\alpha_j)
	=
	\sum_{n_0+\cdots+n_k=N}
	t_1^{n_1}\cdots t_k^{n_k}\,Q_{n_0,\dots,n_k},
\end{equation}
and the expression of $Q_\ell P_\lambda$ is given by the coefficient of $t_1^{\ell_1}\cdots t_k^{\ell_k}$.

For $\sigma\in S_\ell$, let $B^\sigma_\lambda(\alpha) = B^\sigma_\lambda(\alpha_1,\dots,\alpha_k)$ be the coefficients such that
\[
\prod_{j=1}^k E(\alpha_j)P_\lambda = \sum_\sigma B^\sigma_\lambda(\alpha) P_{\lambda+\sigma}.
\]
Let us compute those coefficients now. The Pieri formula applied with $E(\alpha) = \sum_{r=0}^N \alpha^r e_r$ gives
\begin{equation}\label{eq:E-alpha-on-P}
	E(\alpha)P_\lambda
	=
	\sum_{\mu\supset\lambda}
	\alpha^{|\mu|-|\lambda|}\,\psi'_{\mu/\lambda}\,P_\mu,
\end{equation}
where the sum runs over all partitions $\mu$ such that $\mu/\lambda$
is a vertical strip (of arbitrary size).
Applying it successively for
$\alpha_1,\dots,\alpha_k$, we obtain
\begin{equation*}
	\Bigg(\prod_{j=1}^k E(\alpha_j)\Bigg) P_\lambda
	=
	\sum_{\nu\supset\lambda}
		\sum_{\mathcal C}
	\Bigg(\prod_{j=1}^k \alpha_j^{\,|\lambda^{(j)}|-|\lambda^{(j-1)}|}
	\psi'_{\lambda^{(j)}/\lambda^{(j-1)}}\Bigg)\,P_\nu,
\end{equation*}
where the second sum runs over all chains of partitions
\begin{equation*}
	\mathcal C:\quad
	\lambda=\lambda^{(0)}\subset\lambda^{(1)}\subset\cdots\subset\lambda^{(k)}=\nu,
\end{equation*}
such that each skew diagram $\lambda^{(j)}/\lambda^{(j-1)}$ is a
vertical strip (possibly empty). Therefore
\begin{equation}\label{eq:Bsigma}
	B^\sigma_\lambda(\alpha)
	=
	\sum_{\mathcal C}
	\prod_{j=1}^k
	\alpha_j^{|\lambda^{(j)}|-|\lambda^{(j-1)}|}
	\psi'_{\lambda^{(j)}/\lambda^{(j-1)}},
\end{equation}
where the sum runs over all chains
\[
\mathcal C:\qquad
\lambda=\lambda^{(0)}\subset\lambda^{(1)}\subset\cdots\subset\lambda^{(k)}=\lambda+\sigma.
\]

On the other hand, by \eqref{eq:E-Q}, the same coefficient can be written as
\[
B^\sigma_\lambda(\alpha)
=
\sum_{n_0+\cdots+n_k=N}
t_1^{n_1}\cdots t_k^{n_k}\,
A^{(n_1,\dots,n_k)}_\lambda(\sigma).
\]
In particular, $A^\ell_\lambda(\sigma)$ is exactly the coefficient of
$t_1^{\ell_1}\cdots t_k^{\ell_k}$ in this expansion.

\subsection*{Large gap regime.}
The coefficients $A^\ell_\lambda(\sigma)$ do not have a simple expression but they simplify drastically in the large gap regime as the next proposition shows. Recall that
\[
\mathcal A_\ell\de\{\sigma\in S_\ell: \sigma\text{ is an }(\ell_1,\dots,\ell_k)\text{-shape}\}.
\]
\begin{prop}\label{lem:Aasymp}
	For every fixed shape $\sigma\in S_\ell$, one has
	\[
	A^\ell_\lambda(\sigma)
	=
	\ind_{\{\sigma\in\mathcal A_\ell\}}
	+
	O\left(\frac{1}{\gap(\lambda)}\right)
	\qquad\text{as }\gap(\lambda)\to\infty.
	\]
\end{prop}

\begin{proof}
	Let
	\[
	m\de|\sigma|=\sum_{i=1}^N \sigma_i
	\]
	and denote by $\Lambda_{k,m}$ the finite-dimensional space of homogeneous
	symmetric polynomials of degree $m$ in the variables
	$\alpha_1,\dots,\alpha_k$.
	
	We first claim that
	\[
	B^\sigma_\lambda(\alpha_1,\dots,\alpha_k)
	=
	B^0_\sigma(\alpha_1,\dots,\alpha_k)
	+
	R^\sigma_\lambda(\alpha_1,\dots,\alpha_k),
	\]
	where $R^\sigma_\lambda\in \Lambda_{k,m}$ and every coefficient of
	$R^\sigma_\lambda$ in the monomial symmetric basis of $\Lambda_{k,m}$ is
	$O(\gap(\lambda)^{-1})$ and
	\[
	B^{0}_{\sigma}(\alpha_1,\dots,\alpha_k)
	\de
	\sum_{\mathcal C}
	\prod_{j=1}^k \alpha_j^{\,|\lambda^{(j)}|-|\lambda^{(j-1)}|}.
	\]
	Indeed, by Lemma \ref{lem:pierito1}, each Pieri coefficient in
	\eqref{eq:Bsigma} satisfies
	\[
	\psi'_{\lambda^{(j)}/\lambda^{(j-1)}}
	=
	1+O\left(\frac{1}{\gap(\lambda)}\right),
	\]
	uniformly over all admissible chains. Hence for every such chain $\mathcal C$,
	\[
	\prod_{j=1}^k \psi'_{\lambda^{(j)}/\lambda^{(j-1)}}
	=
	1+O\left(\frac{1}{\gap(\lambda)}\right).
	\]
	It follows that each
	coefficient of $R^\sigma_\lambda$ in any fixed basis of
	$\Lambda_{k,m}$ is $O(\gap(\lambda)^{-1})$.
	
	Now assume $\gap(\lambda)\ge k$. Then choosing a chain of partitions ending at $\lambda+\sigma$
	\begin{equation*}
		\mathcal C:\quad
		\lambda=\lambda^{(0)}\subset\lambda^{(1)}\subset\cdots\subset\lambda^{(k)}=\lambda+\sigma,
	\end{equation*}
	is equivalent to choosing, for each row $i$, a subset $S_i\subset[k]$ of cardinality $\sigma_i$, recording at which steps the $i$-row is incremented. Therefore
	\begin{align*}
		\sum_{\mathcal C}
		\prod_{j=1}^k
		\alpha_j^{\,|\lambda^{(j)}|-|\lambda^{(j-1)}|}
		&= \sum_{\substack{S_1,\cdots,S_N\subset [k]\\ |S_i|=\sigma_i}} \alpha_1^{\sum_{j=1}^N \ind_{S_j}(1)}\cdots\alpha_k^{\sum_{j=1}^N \ind_{S_j}(k)}\\
		&= \sum_{\substack{S_1,\cdots,S_N\subset [k]\\ |S_i|=\sigma_i}} \alpha_1^{\ind_{S_1}(1)}\cdots\alpha_k^{\ind_{S_1}(k)}\cdots\alpha_1^{\ind_{S_N}(1)}\cdots\alpha_k^{\ind_{S_N}(k)}\\
		&=\Bigg(\sum_{\substack{S_1\subset [k]\\ |S_1|=\sigma_1}} \alpha_1^{\ind_{S_1}(1)}\cdots\alpha_k^{\ind_{S_1}(k)} \Bigg)\cdots\Bigg(\sum_{\substack{S_N\subset [k]\\ |S_N|=\sigma_N}} \alpha_1^{\ind_{S_N}(1)}\cdots\alpha_k^{\ind_{S_N}(k)}\Bigg),
	\end{align*}
	but $\sum_{\substack{S_i\subset [k]\\ |S_i|=\sigma_i}} \prod_j\alpha_j^{\ind_{S_i}(j)}$ is exactly $e_{\sigma_i}(\alpha)$, thus
	\begin{equation*}
		\sum_{\mathcal C}
		\prod_{j=1}^k
		\alpha_j^{\,|\lambda^{(j)}|-|\lambda^{(j-1)}|}
		=\prod_{i=1}^{N} e_{\sigma_i}(\alpha)
		=\prod_{i=1}^{N} t_{\sigma_i}
		=\prod_{j=1}^{k} t_j^{n_j},
	\end{equation*}
	where $(n_1,\dots,n_k)$ is the shape of $\sigma$. Hence
	\[
	[t_1^{\ell_1}\cdots t_k^{\ell_k}]\,B^0_\sigma
	=
	\ind_{\{\sigma\in\mathcal A_\ell\}}.
	\]
	
	It remains to control the error term after passing from the
	$\alpha$-description to the $t$-description. Notice that the family
	\[
	\Big\{t_1^{a_1}\cdots t_k^{a_k}: a_1,\dots,a_k\ge0, \sum_{r=1}^k r a_r=m\Big\}
	\]
	is a basis of $\Lambda_{k,m}$. Since $\Lambda_{k,m}$ is finite-dimensional, the change of basis is continuous, therefore the coefficients of $R^\sigma_\lambda$ in this basis are also $O(\gap(\lambda)^{-1})$. Finally,
	\begin{align*}
		A^\ell_\lambda(\sigma)
		&=
		[t_1^{\ell_1}\cdots t_k^{\ell_k}]\;B^\sigma_\lambda\\
		&=
		[t_1^{\ell_1}\cdots t_k^{\ell_k}]\;B^0_\sigma
		+
		O\left(\frac{1}{\gap(\lambda)}\right)\\
		&=\ind_{\{\sigma\in\mathcal A_\ell\}}+O\left(\frac{1}{\gap(\lambda)}\right).
	\end{align*}
\end{proof}
\subsection*{Reduction to a finite sum of simple series.}
Inserting \eqref{eq:QPexp} into \eqref{eq:kmoment}, we obtain
\begin{equation*}
	\mathbb E\bigg[\prod_{j=0}^k |c_{n+j}|^{2\ell_j}\bigg]
	=
	\frac{1}{\binom{N}{\ell}^2}\;
	\sum_{\lambda,\nu}\;
	\frac{c_\lambda}{c'_\lambda}\frac{c_\nu}{c'_\nu}\;
	\left|
	\sum_{\sigma\in S_\ell}
	A^\ell_\lambda(\sigma)\,
	\langle P_{\lambda+n+\sigma},P_\nu\rangle_\gamma
	\right|^2.
\end{equation*}
Expanding the square and using orthogonality of the Jack polynomials yields
\begin{equation}\label{eq:kmoment2}
\mathbb E\bigg[\prod_{j=0}^k |c_{n+j}|^{2\ell_j}\bigg]
	=
	\frac{1}{\binom{N}{\ell}^2}\;
	\sum_{\sigma\in S_\ell}\sum_{\lambda}\;
	\frac{c_\lambda}{c'_\lambda}\frac{c_{\lambda+n+\sigma}}{c'_{\lambda+n+\sigma}}
	A^\ell_\lambda(\sigma)^2\;
	\|P_{\lambda+n+\sigma}\|_\gamma^4.
\end{equation}
Using \eqref{lem:normP}, exactly as in the case $k=1$, we can rewrite each inner sum in the form
\[
\sum_{\lambda} a^\sigma_\lambda(n)\,G_\lambda(n),
\]
more precisely
\[
\frac{c_\lambda}{c'_\lambda}\frac{c_{\lambda+n+\sigma}}{c'_{\lambda+n+\sigma}}\;
A^\ell_\lambda(\sigma)^2\,
\|P_{\lambda+n+\sigma}\|_\gamma^4
=
K(N,\gamma)^2\,C(N,\gamma)^2\,a^\sigma_\lambda(n)\,G_\lambda(n),
\]
where $(a^\sigma_\lambda(n))_{\lambda,n}$ is bounded. By Lemma \ref{lem:ratio}, Lemma \ref{lem:Aasymp},
and the same elementary estimate on the product term as in the case $k=1$, we obtain, for every fixed
$\sigma\in S_\ell$,
\[
a^\sigma_\lambda(n)=
\begin{cases}
	1+O\left(\dfrac{1}{\gap(\lambda)}\right),
	& \text{if }\sigma\in\mathcal A_\ell,\\[1em]
	O\left(\dfrac{1}{\gap(\lambda)}\right),
	& \text{if }\sigma\notin\mathcal A_\ell,
\end{cases}
\qquad\text{uniformly in }n.
\]
If $\sigma\in\mathcal A_\ell$, Lemma \ref{lem:asymp} gives
\[
\sum_\lambda a^\sigma_\lambda(n)\,G_\lambda(n)\sim \sum_\lambda G_\lambda(n).
\]
If $\sigma\notin\mathcal A_\ell$, we apply Lemma \ref{lem:asymp} to the family
$1-a^\sigma_\lambda(n)$ and obtain
\[
\sum_\lambda a^\sigma_\lambda(n)\,G_\lambda(n)=o\left(\sum_\lambda G_\lambda(n)\right).
\]
Since $S_\ell$ is finite and $|\mathcal A_\ell|=\binom{N}{\ell}$, it follows from \eqref{eq:kmoment2} that
\begin{align*}
	\mathbb E\bigg[\prod_{j=0}^k |c_{n+j}|^{2\ell_j}\bigg]
	&=
	\frac{1}{\binom{N}{\ell}^2}\,
	K(N,\gamma)^2\,C(N,\gamma)^2
	\left(\binom{N}{\ell}+o(1)\right)\sum_\lambda G_\lambda(n)\\
	&\sim \dfrac{1}{\binom{N}{\ell}}\,\mathbb E|c_n|^{2N}.
\end{align*}
Finally, using the single-mode moment asymptotics obtained in Section \ref{sec:mom1},
\begin{align*}
\mathbb E\bigg[\prod_{j=0}^k |c_{n+j}|^{2\ell_j}\bigg]
	\sim
	\ell_0!\cdots \ell_k!\;\kappa(\beta)^N\,n^{-N(1-2\gamma)},
	\qquad\text{as }n\rightarrow\infty.
\end{align*}
The constant term is exactly the mixed modulus moment of $k+1$ independent complex Gaussian random variables
with law $\mathcal N_{\mathbb C}(0,\kappa(\beta))$.

\subsection{Complex mixed moments}\label{sub:mix}
Let us now look at the general moments where some exponents of $c_{n+j}$ do not match those of $\overline{c_{n+j}}$:
\begin{equation*}
	M_n(\ell,m)\de \E{\prod_{j=0}^k c_{n+j}^{\ell_j}\,\overline{c_{n+j}}^{m_j}},
\end{equation*}
where $\ell_j,m_j$ are nonnegative integers, and set
\begin{equation*}
	d_j \de \ell_j-m_j,\qquad N_+\de \sum_{j=0}^k \ell_j,\qquad N_-\de \sum_{j=0}^k m_j.
\end{equation*}
Introducing angle variables $\theta_{j,a}$ for $1\le a\le \ell_j$ and $\theta'_{j,b}$ for $1\le b\le m_j$, a direct computation yields the Coulomb-gas integral
\begin{align}\label{eq:cmix}
	M_n(\ell,m)
	&=\frac{1}{(2\pi)^{N_-+N_+}}\int_{[0,2\pi]^{N_++N_-}}
	\exp\Bigg(-\ii\sum_{j=0}^k (n+j)\bigg\{\sum_{a=1}^{\ell_j}\theta_{j,a}-\sum_{b=1}^{m_j}\theta'_{j,b}\bigg\}\Bigg)\nonumber\\
	&\qquad\times
	\frac{
		\prod_{(j,a)<(j',a')}|\e^{\ii\theta_{j,a}}-\e^{\ii\theta_{j',a'}}|^{2\gamma}\;
		\prod_{(j,b)<(j',b')}|\e^{\ii\theta'_{j,b}}-\e^{\ii\theta'_{j',b'}}|^{2\gamma}
	}{
		\prod_{(j,a),(j',b)}|\e^{\ii\theta_{j,a}}-\e^{\ii\theta'_{j',b}}|^{2\gamma}
	}\;
	\dth\;\dthp,
\end{align}
where $<$ is the lexicographic order. Now, let 
\begin{equation*}
	S_n(\ell,m)\de\sum_{j=0}^k (n+j)\,d_j.
\end{equation*}
If $S_n(\ell,m)\neq 0$, then $M_n(\ell,m)=0$. Indeed, if one performs the change of variables for all angles
\begin{equation*}
	\theta_{j,a}\mapsto \theta_{j,a}+\alpha,\qquad \theta'_{j,b}\mapsto \theta'_{j,b}+\alpha,
\end{equation*}
the Coulomb-gas factor in \eqref{eq:cmix} depends only on differences of angles, hence is invariant, and the oscillatory factor gets an extra phase $\e^{-\ii\alpha S_n(\ell,m)}$. Therefore, for every $\alpha$,
\begin{equation*}
	M_n(\ell,m)=\e^{-\ii\alpha S_n(\ell,m)}\,M_n(\ell,m).
\end{equation*}
Choosing $\alpha$ so that $\e^{-\ii\alpha S_n(\ell,m)}\neq 1$ forces $M_n(\ell,m)=0$.

\noindent If $S_n(\ell,m)=0$ for some large enough $n$, then necessarily
\begin{equation*}
	\sum_{j=0}^k d_j=0	\qquad\text{and}\qquad\sum_{j=0}^k j\,d_j=0.
\end{equation*}
The following proposition treats this particular case.
\begin{prop}
	Assume that $\sum d_j=\sum j d_j=0$ and $\ell\neq m$, then $N=N_+=N_-$ and
	\begin{equation*}
		n^{N(1-\beta^2)} M_n(\ell,m)\longrightarrow 0, \qquad \text{as }n\rightarrow \infty.
	\end{equation*}
\end{prop}

\begin{proof}
Using Stanley's Cauchy identity twice, one obtains a double series analogous to the previous section, but with additional insertions:
\begin{equation*}
M_n(\ell,m)
=C_{\ell,m}\;\sum_{\lambda,\nu}
\frac{c_\lambda}{c'_\lambda}\frac{c_\nu}{c'_\nu}
\,
\langle Q_\ell P_{\lambda+n},P_\nu\rangle_\gamma
\,
\langle Q_m P_{\lambda+n},P_\nu\rangle_\gamma,
\end{equation*}
where $C_{\ell,m}= \binom{N}{\ell}^{-1}\binom{N}{m}^{-1}$ and
\begin{equation*}
	Q_{\ell}(x)
	=
	\sum_{\substack{(I_0,\dots,I_k)\ \text{partition of }\{1,\dots,N\}\\ |I_j|=\ell_j}}
	\prod_j x_{I_j}^{j}.
\end{equation*}
Recall the decomposition \eqref{eq:QPexp} of the previous section
\[
Q_\ell P_{\lambda+n}
=
\sum_{\sigma\in S_\ell} A^\ell_\lambda(\sigma)\,P_{\lambda+n+\sigma}.
\]

Replacing the first scalar product in the expression of $M_n(\ell,m)$ by this decomposition, and
interchanging the sums over  $\nu$ and  $\sigma$, we obtain
\begin{align*}
	M_n(\ell,m)
	&=C_{\ell,m}
	\sum_{\sigma\in S_\ell}\sum_{\lambda,\nu}
	\frac{c_\lambda}{c'_\lambda}\frac{c_\nu}{c'_\nu}\;
	A^\ell_\lambda(\sigma)\,
	\langle P_{\lambda+n+\sigma},P_\nu\rangle_\gamma
	\langle Q_m P_{\lambda+n},P_\nu\rangle_\gamma.
\end{align*}
By orthogonality of the Jack polynomials, only the term  $\nu=\lambda+n+\sigma$ survives, so that
\begin{align*}
	M_n(\ell,m)
	&=C_{\ell,m}
	\sum_{\sigma\in S_\ell}\sum_{\lambda}
	\frac{c_\lambda}{c'_\lambda}\frac{c_{\lambda+n+\sigma}}{c'_{\lambda+n+\sigma}}\;
	A^\ell_\lambda(\sigma)\,
	\langle Q_m P_{\lambda+n},P_{\lambda+n+\sigma}\rangle_\gamma
	\|P_{\lambda+n+\sigma}\|_\gamma^2 .
\end{align*}
Applying the same decomposition to  $Q_m P_{\lambda+n}$, we get
\[
\langle Q_m P_{\lambda+n},P_{\lambda+n+\sigma}\rangle_\gamma
=
A^m_\lambda(\sigma)\,\|P_{\lambda+n+\sigma}\|_\gamma^2,
\]
and therefore
\begin{equation}\label{eq:Mnserie}
	M_n(\ell,m)
	=C_{\ell,m}
	\sum_{\sigma\in S_\ell}\sum_{\lambda}
	\frac{c_\lambda}{c'_\lambda}\frac{c_{\lambda+n+\sigma}}{c'_{\lambda+n+\sigma}}\;
	A^\ell_\lambda(\sigma)A^m_\lambda(\sigma)
	\|P_{\lambda+n+\sigma}\|_\gamma^4.
\end{equation}
By Proposition \ref{lem:Aasymp}, one has
\[
A^\ell_\lambda(\sigma)
=
\ind_{\{\sigma\in \mathcal A_\ell\}}
+O\left(\frac1{\gap(\lambda)}\right),
\qquad
A^m_\lambda(\sigma)
=
\ind_{\{\sigma\in \mathcal A_m\}}
+O\left(\frac1{\gap(\lambda)}\right).
\]
Recall now that  $\ell\neq m $. Thus $\mathcal A_\ell\cap\mathcal A_m=\varnothing$ and it follows that, for every  $\sigma\in S_\ell$,
\[
A^\ell_\lambda(\sigma)A^m_\lambda(\sigma)
=
O \left(\frac1{\gap(\lambda)}\right).
\]
By reductions similar to those in Section \ref{sec:generalk}, each
inner sum in \eqref{eq:Mnserie} may be written in the form
\[
\sum_\lambda a^\sigma_\lambda(n)\,G_\lambda(n),
\]
where  $(a^\sigma_\lambda(n))_{\lambda,n}$ is bounded and satisfies
\[
a^\sigma_\lambda(n)\longrightarrow 0
\qquad\text{as }\gap(\lambda)\to\infty,
\]
uniformly in  $n$. By the same argument as in Lemma \ref{lem:asymp}, this implies
\[
\sum_\lambda a^\sigma_\lambda(n)\,G_\lambda(n)
=
o \left(\sum_\lambda G_\lambda(n)\right)
=
o \left(n^{-N(1-2\gamma)}\right), \qquad \text{ as } n\rightarrow\infty.
\]
Since  $S_\ell$ is finite, summing over  $\sigma\in S_\ell$ finally gives
\[
M_n(\ell,m)=o \left(n^{-N(1-2\gamma)}\right), \qquad \text{ as } n\rightarrow\infty.
\]
This completes the proof.
\end{proof}

Combining the results of this section and the previous one, all mixed moments converge to those of independent complex Gaussians $\mathcal N_\C(0,\kappa(\beta))$. Since the limiting law is moment-determinate, this proves the finite-dimensional convergence in Theorem \ref{thm:CLT}.

\section{Convergence toward complex white noise}\label{sec:whitenoise}
We now prove Theorem \ref{thm:whitenoise}. Fix $s<-\frac{1}{2}$ and choose $s'$ such that $s<s'<-\frac{1}{2}$.
Define the rescaled random distributions
\begin{equation*}
	X_n \de n^{\frac{1-\beta^2}{2}}\,e^{\ii n\theta}\,\M \in \mathcal D'(\mathbb T),
\end{equation*}
whose Fourier coefficients are
\begin{equation*}
	\widehat{X_n}(k)
	= n^{\frac{1-\beta^2}{2}}\,c_{k-n}.
\end{equation*}

\subsection*{Tightness in $H^s$}
We claim that
\begin{equation}\label{eq:tight}
	\sup_{n\in\N}\, \E{\norm{X_n}_{H^{s'}}^2}<+\infty.
\end{equation}
Indeed,
\begin{equation*}
	 \E{\norm{X_n}_{H^{s'}}^2}
	=\sum_{k\in\Z}\left(1+k^2\right)^{s'}\,\mathbb E|\widehat{X_n}(k)|^2
	= n^{1-\beta^2}\sum_{k\in\Z}\left(1+k^2\right)^{s'}\,\mathbb E|c_{k-n}|^2.
\end{equation*}
Using the fact that for some $C>0$, one has
\begin{equation*}
	\mathbb E|c_m|^2 \leq C\; (1+|m|)^{-(1-\beta^2)},\qquad m\in\Z,
\end{equation*}
and changing variables $m=k-n$ gives
\begin{equation*}
\E{\norm{X_n}_{H^{s'}}^2}\leq C\; n^{1-\beta^2}\sum_{m\in\Z}(1+(m+n)^2)^{s'}(1+|m|)^{-(1-\beta^2)}.
\end{equation*}
Splitting the sum into $|m|\ge n/2$ and $|m|<n/2$ yields, for some constant $C_1>0$,
\begin{equation*}
 \E{\norm{X_n}_{H^{s'}}^2}
\leq C_1\; \bigg(\sum_{j\in\Z}(1+j^2)^{s'} +n^{1+2s'}\bigg),
\end{equation*}
which is bounded in $n$ since $s'<-1/2$. This proves \eqref{eq:tight}.
Since the embedding $H^{s'}(\mathbb T)\hookrightarrow H^s(\mathbb T)$ is compact for $s<s'$,
we conclude that $(X_n)$ is tight in $H^s(\mathbb T)$.

\subsection*{Identification of the limit.}
Let $\phi(\theta)=\sum_{|k|\le K} a_k e^{\ii k\theta}$ be a trigonometric polynomial, one has
\begin{equation*}
	\langle X_n,\phi\rangle
	= 2\pi \,n^{\frac{1-\beta^2}{2}}\sum_{|k|\le K}a_k\,c_{-n-k}.
\end{equation*}
By the joint convergence of Theorem \ref{thm:CLT} (the same convergence holds for
negative frequencies, by reflection invariance of the field), we have the convergence in distribution
\begin{equation*}
n^{\frac{1-\beta^2}{2}}(c_{-n-K},\ldots,c_{-n+K})
\longrightarrow (\xi_{-K},\ldots,\xi_K),\quad\text{as }n\rightarrow\infty.
\end{equation*}
where $(\xi_k)_{k\in\Z}$ are i.i.d. complex Gaussians $\mathcal N_{\C}(0,\kappa(\beta))$.
Therefore, in distribution,
\begin{equation*}
	\langle X_n,\phi\rangle \longrightarrow2\pi \sum_{|k|\le K}a_k\,\xi_{-k}.
\end{equation*}
Define the limiting random distribution
\begin{equation*}
	W \de \sum_{k\in\Z}\xi_k\,\e^{\ii k\theta},
\end{equation*}
which lives in $H^s(\mathbb T)$ since
\begin{equation*}
\E{\norm{W}_{H^s}^2}=\kappa(\beta)\sum_{k\in\Z}(1+k^2)^s<+\infty.
\end{equation*}
We have $\langle X_n,\phi\rangle\longrightarrow \langle W,\phi\rangle$ in distribution.

Combining these two results yields Theorem \ref{thm:whitenoise}.

\subsection*{Statements and Declarations}

\subsubsection*{Acknowledgements}
The authors would like to thank Baptiste Cerclé for very stimulating discussions.

\subsubsection*{Funding}
This work was supported by the Swiss National Science Foundation (SNSF) grant 207410, ``2d constructive field theory with exponential interactions''.

\subsubsection*{Competing interests}
The authors have no competing interests to declare that are relevant to the content of this article.

\subsubsection*{Data availability}
Data sharing not applicable to this article as no datasets were generated or analysed during the current study.

\appendix
\section{Partitions and Young diagrams}\label{partition}
A partition is a finite non-increasing sequence of nonnegative integers
$\lambda=(\lambda_1,\lambda_2,\ldots)$ with $\lambda_i=0$ for $i$ large enough.
Its size is $|\lambda|\de\sum_{i\ge1}\lambda_i$, and its length
$\ell(\lambda)$ is the number of positive parts. In the present work, we identify partitions of length $\ell(\lambda)\leq N$ with non-increasing sequences of $N$ integers so that $\lambda=(\lambda_1,\dots,\lambda_N)$.

The Young diagram of $\lambda$ is the set of boxes
\begin{equation*}
	\{ (i,j)\in\N^2 : 1\le i\le \ell(\lambda) \text{ and } 1\le j\le \lambda_i \},
\end{equation*}
with rows indexed from top to bottom and columns from left to right.
The conjugate partition $\lambda'$ is defined by $\lambda'_j=|\{\,i:\lambda_i\ge j\,\}|$ and corresponds to reflecting the diagram across the main diagonal.

For a cell $s=(i,j)$ in the diagram, the arm-length and leg-length are defined by
\begin{equation*}
		a_\lambda(s)\de \lambda_i-j,\qquad l_\lambda(s)\de \lambda'_j-i.
\end{equation*}

\begin{figure}[H]
	\centering
	\subfigure{
\begin{tikzpicture}[scale=0.45]

	\def\SelRow{2}
	\def\SelCol{7}
	
	\newcount\SelRowLenCount
	\newcount\LegCount
	\SelRowLenCount=0
	\LegCount=0
	
	\foreach \RowIndex/\RowLength in {0/18,1/16,2/13,3/11,4/9,5/7,6/6,7/4}{
		\ifnum\RowIndex=\SelRow\relax
		\SelRowLenCount=\RowLength\relax
		\fi

		\ifnum\RowIndex>\SelRow\relax
		\ifnum\RowLength>\SelCol\relax 
		\advance\LegCount by 1\relax
		\fi
		\fi
	}

	\pgfmathtruncatemacro{\Arm}{\the\SelRowLenCount - 1 - \SelCol}
	
	\foreach \row/\RowLength in {0/18,1/16,2/13,3/11,4/9,5/7,6/6,7/4}{
		\foreach \x in {0,...,\numexpr\RowLength-1\relax}{
			
			\def\BoxFill{white}
			
			\ifnum\x=\SelCol\relax
			\ifnum\row>\SelRow\relax
			\def\BoxFill{red!25}
			\fi
			\fi
			
			\ifnum\row=\SelRow\relax
			\ifnum\x>\SelCol\relax
			\def\BoxFill{blue!20}
			\fi
			\fi
			
			\ifnum\row=\SelRow\relax
			\ifnum\x=\SelCol\relax
			\def\BoxFill{yellow!35}
			\fi
			\fi
			
			\fill[\BoxFill] (\x,-\row) rectangle ++(1,-1);
			\draw[thick] (\x,-\row) rectangle ++(1,-1);
			
			\ifnum\row=\SelRow\relax
			\ifnum\x=\SelCol\relax
			\node at (\x+0.5,-\row-0.5) {\large $s$};
			\fi
			\fi
		}
	}
	
	\node at (10,1.2) {\Large $\lambda$};
	\end{tikzpicture}
	}
	\caption{The Young diagram of a partition $\lambda$ with $a_\lambda(s)=5$ and $l_\lambda(s)=2$.}
	\label{fig:+armleg}
	\end{figure}
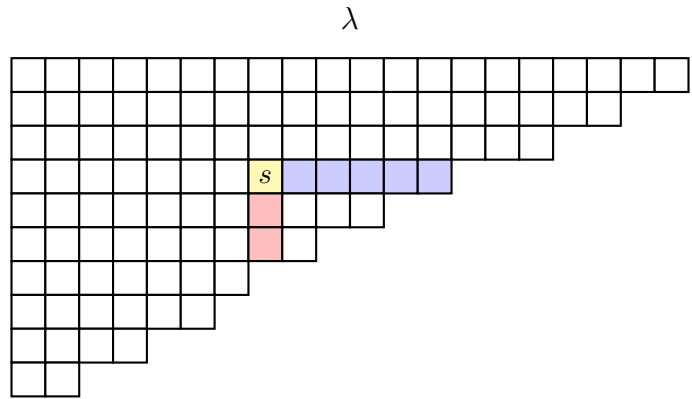
	
If $\lambda\subset\mu$ (i.e. $\lambda_i\le \mu_i$ for all $i$), the skew diagram
$\mu/\lambda$ is the set-theoretic difference of their Young diagrams.  A skew diagram $\mu/\lambda$ is a horizontal strip if it contains at most one cell in each column, and a vertical strip if it contains at most one cell in each row. 

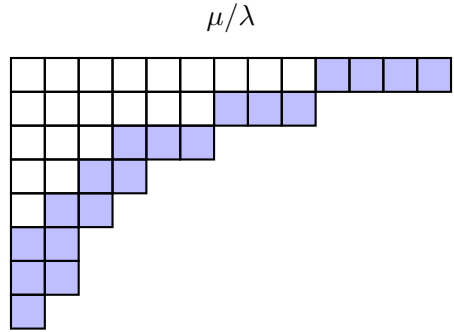
\begin{figure}[H]
	\centering
	\subfigure{

\begin{tikzpicture}[scale=0.45]
	
	\foreach \row/\MuLen in {0/13,1/9,2/6,3/4,4/3,5/2,6/2,7/1}{
		\foreach \x in {0,...,\numexpr\MuLen-1\relax}{
			\fill[blue!25] (\x,-\row) rectangle ++(1,-1);
		}
	}

	\foreach \row/\LamLen in {0/9,1/6,2/3,3/2,4/1}{
		\foreach \x in {0,...,\numexpr\LamLen-1\relax}{
			\fill[white] (\x,-\row) rectangle ++(1,-1);
		}
	}
	
	\foreach \row/\MuLen in {0/13,1/9,2/6,3/4,4/3,5/2,6/2,7/1}{
		\foreach \x in {0,...,\numexpr\MuLen-1\relax}{
			\draw[thick] (\x,-\row) rectangle ++(1,-1);
		}
	}

	\node at (6.5,1.2) {\large $\mu/\lambda$};
	
	\end{tikzpicture}
	}
		\caption{The skew diagram obtained from two partitions $\lambda$ and $\mu$ represented by the blue cells.}
	\label{fig:+skew}
	\end{figure}

A $p$-vertical strip is a vertical strip with $p$ elements. We use the notation and conventions of \cite[Chapter I]{Macdonald95}, where further background on these notions can be found.

\bibliographystyle{plain}
\bibliography{biblio_clean}

@preamble{ "\def\cprime{$'$} "
}

@article{kahane85,
  author =        {Kahane, Jean-Pierre},
  journal =       {Ann. Sci. Math. Qu\'ebec},
  number =        {2},
  pages =         {105--150},
  title =         {Sur le chaos multiplicatif},
  volume =        {9},
  year =          {1985},
  issn =          {0707-9109},
}

@article{mandelbrot74,
  author =        {Mandelbrot, Benoit},
  journal =       {C. R. Acad. Sci. Paris S\'er. A},
  pages =         {289--292},
  title =         {Multiplications al\'eatoires it\'er\'ees et
                   distributions invariantes par moyenne pond\'er\'ee
                   al\'eatoire},
  volume =        {278},
  year =          {1974},
  issn =          {0302-8429},
}

@article{hoeghkrohn71,
  author =        {H{\o}egh-Krohn, Raphael},
  journal =       {Comm. Math. Phys.},
  pages =         {244--255},
  title =         {A general class of quantum fields without cut-offs in
                   two space-time dimensions},
  volume =        {21},
  year =          {1971},
  issn =          {0010-3616,1432-0916},
  url =           {http://projecteuclid.org/euclid.cmp/1103857337},
}

@article{berestycki17,
  author =        {Nathana{\"e}l Berestycki},
  journal =       {Electronic Communications in Probability},
  number =        {none},
  pages =         {1 -- 12},
  publisher =     {Institute of Mathematical Statistics and Bernoulli
                   Society},
  title =         {{An elementary approach to Gaussian multiplicative
                   chaos}},
  volume =        {22},
  year =          {2017},
  doi =           {10.1214/17-ECP58},
  url =           {https://doi.org/10.1214/17-ECP58},
}

@article{SHAMOV16,
  author =        {Alexander Shamov},
  journal =       {Journal of Functional Analysis},
  number =        {9},
  pages =         {3224-3261},
  title =         {On Gaussian multiplicative chaos},
  volume =        {270},
  year =          {2016},
  doi =           {https://doi.org/10.1016/j.jfa.2016.03.001},
  issn =          {0022-1236},
  url =           {https://www.sciencedirect.com/science/article/pii/
                  S0022123616000987},
}

@article{rhodesvargas14,
  author =        {Rhodes, R\'{e}mi and Vargas, Vincent},
  journal =       {Probab. Surv.},
  pages =         {315--392},
  title =         {Gaussian multiplicative chaos and applications: a
                   review},
  volume =        {11},
  year =          {2014},
  doi =           {10.1214/13-PS218},
  issn =          {1549-5787},
  url =           {https://doi.org/10.1214/13-PS218},
}

@article{bertacco23,
  author =        {Bertacco, Federico},
  journal =       {Electronic Journal of Probability},
  number =        {none},
  pages =         {1 -- 36},
  publisher =     {Institute of Mathematical Statistics and Bernoulli
                   Society},
  title =         {{Multifractal analysis of Gaussian multiplicative
                   chaos and applications}},
  volume =        {28},
  year =          {2023},
  doi =           {10.1214/22-EJP893},
  url =           {https://doi.org/10.1214/22-EJP893},
}

@article{lacoin24,
  author =        {Hubert Lacoin},
  journal =       {Annales de l'Institut Henri Poincaré, Probabilités
                   et Statistiques},
  number =        {4},
  pages =         {2328 -- 2351},
  publisher =     {Institut Henri Poincaré},
  title =         {{Critical Gaussian multiplicative chaos revisited}},
  volume =        {60},
  year =          {2024},
  doi =           {10.1214/23-AIHP1411},
  url =           {https://doi.org/10.1214/23-AIHP1411},
}

@article{bertaccohairer25,
  author =        {Bertacco, Federico and Hairer, Martin},
  journal =       {Electronic Journal of Probability},
  number =        {none},
  pages =         {1 -- 22},
  publisher =     {Institute of Mathematical Statistics and Bernoulli
                   Society},
  title =         {{Uniqueness of supercritical Gaussian multiplicative
                   chaos}},
  volume =        {30},
  year =          {2025},
  doi =           {10.1214/25-EJP1449},
  url =           {https://doi.org/10.1214/25-EJP1449},
}

@article{des93,
  author =        {Derrida, B. and Evans, M. R. and Speer, E. R.},
  journal =       {Comm. Math. Phys.},
  number =        {2},
  pages =         {221--244},
  title =         {Mean field theory of directed polymers with random
                   complex weights},
  volume =        {156},
  year =          {1993},
  issn =          {0010-3616},
  url =           {http://projecteuclid.org/euclid.cmp/1104253626},
}

@article{barraljinmandelbrot10,
  author =        {Barral, Julien and Jin, Xiong and
                   Mandelbrot, Beno{\^{\i}}t},
  journal =       {The Annals of Applied Probability},
  number =        {4},
  pages =         {1219 -- 1252},
  publisher =     {Institute of Mathematical Statistics},
  title =         {Convergence of complex multiplicative cascades},
  volume =        {20},
  year =          {2010},
  doi =           {10.1214/09-AAP665},
  url =           {https://doi.org/10.1214/09-AAP665},
}

@article{lacoin22subchaos,
  author =        {Lacoin, Hubert},
  journal =       {Ann. Appl. Probab.},
  number =        {1},
  pages =         {269--293},
  title =         {A universality result for subcritical complex
                   {G}aussian multiplicative chaos},
  volume =        {32},
  year =          {2022},
  doi =           {10.1214/21-aap1677},
  issn =          {1050-5164,2168-8737},
  url =           {https://doi.org/10.1214/21-aap1677},
}

@article{lacoin22,
  author =        {Hubert Lacoin},
  journal =       {The Annals of Probability},
  number =        {3},
  pages =         {950 -- 983},
  publisher =     {Institute of Mathematical Statistics},
  title =         {{Convergence in law for complex Gaussian
                   multiplicative chaos in phase III}},
  volume =        {50},
  year =          {2022},
  doi =           {10.1214/21-AOP1551},
  url =           {https://doi.org/10.1214/21-AOP1551},
}

@article{lrv2015,
  author =        {Lacoin, Hubert and Rhodes, R{\'e}mi and
                   Vargas, Vincent},
  journal =       {Comm. Math. Phys.},
  number =        {2},
  pages =         {569--632},
  title =         {Complex {G}aussian multiplicative chaos},
  volume =        {337},
  year =          {2015},
  doi =           {10.1007/s00220-015-2362-4},
  issn =          {0010-3616},
  url =           {http://dx.doi.org/10.1007/s00220-015-2362-4},
}

@article{JunnilaSaksmanWebb20,
  author =        {Janne Junnila and Eero Saksman and Christian Webb},
  journal =       {The Annals of Applied Probability},
  number =        {5},
  pages =         {2099 -- 2164},
  publisher =     {Institute of Mathematical Statistics},
  title =         {{Imaginary multiplicative chaos: Moments, regularity
                   and connections to the Ising model}},
  volume =        {30},
  year =          {2020},
  doi =           {10.1214/19-AAP1553},
  url =           {https://doi.org/10.1214/19-AAP1553},
}

@article{arubaverezjegojunnila25,
  author =        {Juhan Aru and Guillaume Baverez and Antoine Jego and
                   Janne Junnila},
  journal =       {Electronic Journal of Probability},
  month =         {1},
  number =        {none},
  pages =         {1--43},
  title =         {Noise-like analytic properties of imaginary chaos},
  volume =        {30},
  year =          {2025},
}

@article{Arujegojunnila22,
  author =        {Aru, Juhan and Jego, Antoine and Junnila, Janne},
  journal =       {Probab. Theory Related Fields},
  number =        {3-4},
  pages =         {749--803},
  title =         {Density of imaginary multiplicative chaos via
                   {M}alliavin calculus},
  volume =        {184},
  year =          {2022},
  doi =           {10.1007/s00440-022-01135-y},
  issn =          {0178-8051,1432-2064},
  url =           {https://doi.org/10.1007/s00440-022-01135-y},
}

@article{Frohlich76,
  author =        {Fr{\"o}hlich, J{\"u}rg},
  journal =       {Communications in Mathematical Physics},
  number =        {3},
  pages =         {233--268},
  title =         {Classical and quantum statistical mechanics in one
                   and two dimensions: Two-component Yukawa ---and
                   Coulomb systems},
  volume =        {47},
  year =          {1976},
  doi =           {10.1007/BF01609843},
  isbn =          {1432-0916},
  url =           {https://doi.org/10.1007/BF01609843},
}

@article{lrv23,
  author =        {Lacoin, Hubert and Rhodes, R{\'e}mi and
                   Vargas, Vincent},
  journal =       {Probability Theory and Related Fields},
  number =        {1},
  pages =         {1--40},
  title =         {A probabilistic approach of ultraviolet
                   renormalization in the boundary Sine-Gordon model},
  volume =        {185},
  year =          {2023},
  doi =           {10.1007/s00440-022-01174-5},
  isbn =          {1432-2064},
  url =           {https://doi.org/10.1007/s00440-022-01174-5},
}

@article{fendleylesagesaleur95,
  author =        {Fendley, Paul and Lesage, Frederic and
                   Saleur, Hubert},
  journal =       {J. Statist. Phys.},
  number =        {5-6},
  pages =         {799--819},
  title =         {Solving {$1$}D plasmas and {$2$}D boundary problems
                   using {J}ack polynomials and functional relations},
  volume =        {79},
  year =          {1995},
  doi =           {10.1007/BF02181204},
  issn =          {0022-4715,1572-9613},
  url =           {https://doi.org/10.1007/BF02181204},
}

@article{UsciatiGuillarmouRhodesSantachiara26,
  author =        {Usciati, Romain and Guillarmou, Colin and
                   Rhodes, Remi and Santachiara, Raoul},
  journal =       {Phys. Rev. Lett.},
  month =         {Jan},
  pages =         {031601},
  publisher =     {American Physical Society},
  title =         {Probabilistic Construction of Noncompactified
                   Imaginary Liouville Field Theory},
  volume =        {136},
  year =          {2026},
  doi =           {10.1103/d723-gk4m},
  url =           {https://link.aps.org/doi/10.1103/d723-gk4m},
}

@article{garban2024harmonicanalysisgaussianmultiplicative,
  author =        {Christophe Garban and Vincent Vargas},
  title =         {Harmonic analysis of Gaussian multiplicative chaos on
                   the circle},
  year =          {2024},
  url =           {https://arxiv.org/abs/2311.04027},
}

@misc{chen2025harmonicanalysismandelbrotcascades,
  author =        {Xinxin Chen and Yong Han and Yanqi Qiu and
                   Zipeng Wang},
  title =         {Harmonic analysis of Mandelbrot cascades -- in the
                   context of vector-valued martingales},
  year =          {2025},
  url =           {https://arxiv.org/abs/2409.13164},
}

@article{lin2025harmonicanalysismultiplicativechaos,
  author =        {Zhaofeng Lin and Yanqi Qiu and Mingjie Tan},
  title =         {Harmonic analysis of multiplicative chaos Part I: the
                   proof of Garban-Vargas conjecture for 1D GMC},
  year =          {2025},
  url =           {https://arxiv.org/abs/2411.13923},
}

@misc{lin2025harmonicanalysismultiplicativechaosII,
  author =        {Zhaofeng Lin and Yanqi Qiu and Mingjie Tan},
  title =         {Harmonic analysis of multiplicative chaos Part II: a
                   unified approach to Fourier dimensions},
  year =          {2025},
  url =           {https://arxiv.org/abs/2505.03298},
}

@misc{chen2025exactvaluesfourierdimensions,
  author =        {Yukun Chen and Zhaofeng Lin and Yanqi Qiu},
  title =         {Exact values of Fourier dimensions of Gaussian
                   multiplicative chaos on high dimensional torus},
  year =          {2025},
  url =           {https://arxiv.org/abs/2507.23494},
}

@misc{arguin2026fouriercoefficientscriticalgaussian,
  author =        {Louis-Pierre Arguin and Jad Hamdan},
  title =         {On the Fourier coefficients of critical Gaussian
                   multiplicative chaos},
  year =          {2026},
  url =           {https://arxiv.org/abs/2510.24424},
}

@article{najnudelpaquettesimm23,
  author =        {Joseph Najnudel and Elliot Paquette and Nick Simm},
  journal =       {The Annals of Probability},
  number =        {4},
  pages =         {1193 -- 1248},
  publisher =     {Institute of Mathematical Statistics},
  title =         {{Secular coefficients and the holomorphic
                   multiplicative chaos}},
  volume =        {51},
  year =          {2023},
  doi =           {10.1214/22-AOP1616},
  url =           {https://doi.org/10.1214/22-AOP1616},
}

@misc{najnudel2025fouriercoefficientsholomorphicmultiplicative,
  author =        {Joseph Najnudel and Elliot Paquette and Nick Simm and
                   Truong Vu},
  title =         {The Fourier coefficients of the holomorphic
                   multiplicative chaos in the limit of large frequency},
  year =          {2025},
  url =           {https://arxiv.org/abs/2502.14863},
}

@misc{atherfold2025fouriercoefficientscriticalholomorphic,
  author =        {Christopher Atherfold and Joseph Najnudel},
  title =         {The Fourier coefficients of the critical holomorphic
                   multiplicative chaos},
  year =          {2025},
  url =           {https://arxiv.org/abs/2508.13849},
}

@misc{junnilasaksmanviitasaari2019,
  author =        {Janne Junnila and Eero Saksman and Lauri Viitasaari},
  title =         {On the regularity of complex multiplicative chaos},
  year =          {2019},
  url =           {https://arxiv.org/abs/1905.12027},
}

@article{JSW20,
  author =        {Janne Junnila and Eero Saksman and Christian Webb},
  journal =       {The Annals of Applied Probability},
  number =        {6},
  pages =         {3786 -- 3820},
  publisher =     {Institute of Mathematical Statistics},
  title =         {{Decompositions of log-correlated fields with
                   applications}},
  volume =        {29},
  year =          {2019},
  doi =           {10.1214/19-AAP1492},
  url =           {https://doi.org/10.1214/19-AAP1492},
}

@article{lebleserfatyzeitouni17,
  author =        {Lebl\'e, Thomas and Serfaty, Sylvia and
                   Zeitouni, Ofer},
  journal =       {Comm. Math. Phys.},
  number =        {1},
  pages =         {301--360},
  title =         {Large deviations for the two-dimensional
                   two-component plasma},
  volume =        {350},
  year =          {2017},
  doi =           {10.1007/s00220-016-2735-3},
  issn =          {0010-3616,1432-0916},
  url =           {https://doi.org/10.1007/s00220-016-2735-3},
}

@book{Macdonald95,
  author =        {Macdonald, Ian Grant},
  month =         {03},
  publisher =     {Oxford University Press},
  title =         {Symmetric Functions and Hall Polynomials},
  year =          {1995},
  doi =           {10.1093/oso/9780198534891.001.0001},
  isbn =          {9780198534891},
  url =           {https://doi.org/10.1093/oso/9780198534891.001.0001},
}

@article{ForresterWarnaar08,
  author =        {Forrester, Peter and Warnaar, Ole},
  journal =       {Bulletin of the American Mathematical Society},
  month =         {10},
  pages =         {489–534},
  title =         {The importance of the Selberg integral},
  volume =        {45},
  year =          {2008},
  doi =           {10.1090/S0273-0979-08-01221-4},
}

@book{SchmeisserTriebel87,
  author =        {Schmeisser, Hans-J\"urgen and Triebel, Hans},
  pages =         {300},
  publisher =     {John Wiley \& Sons, Ltd., Chichester},
  series =        {A Wiley-Interscience Publication},
  title =         {Topics in {F}ourier analysis and function spaces},
  year =          {1987},
  isbn =          {0-471-90895-9},
}

 \end{document}